\documentclass[11pt]{amsart}
\usepackage{amsfonts}

\theoremstyle{plain}
\newtheorem{acknowledgement}{Acknowledgement}

\newtheorem{corollary}{Corollary}

\newtheorem{definition}{Definition}
\newtheorem{example}{Example}

\newtheorem{lemma}{Lemma}

\newtheorem{proposition}{Proposition}
\newtheorem{remark}{Remark}

\numberwithin{equation}{section}
\input{tcilatex}

\begin{document}
\title[Skorokhod-Malliavin Calculus]{On distribution free
Skorokhod-Malliavin calculus}
\author{R. Mikulevicius}
\address{University of Southern California}
\email{mikulvcs@math.usc.edu}
\author{B.L. Rozovskii}
\thanks{B.L. Rozovskii acknowledges support by ARO grant W911NF-13-1-0012. }
\address{Brown University}
\email{boris\_rozovsky@brown.edu}
\date{June 20, 2014}
\subjclass[2000]{Primary 60H05, 60H07; Secondary, 60H10, 60H15 }
\keywords{Distribution free Skorokhod-Malliavin calculus; linear SDEs and
SPDEs}

\begin{abstract}
The starting point of the current paper is a sequence of uncorrelated random
variables. The distribution functions\ of these variables are assumed to be
given but no assumptions on the types or the structure of these
distributions are made. The above setting constitute the so called
"distribution free" paradigm. Under these assumptions, a version of
Skorokhod-Malliavin calculus is developed and applications to stochastic
PDES are discussed.
\end{abstract}

\maketitle

\section{Introduction}

\bigskip The theory and applications of Skorokhod-Malliavin calculus are
well developed for Gaussian and Poisson processes, see, for example, A.V.
Skorokhod \cite{sk}, P. Malliavin \cite{Mal1}, \cite{Mal2}, G. Di Nunno et
al \cite{Ok2}, D. Nualart \cite{Nualrt}, M. Sanz-Sole \cite{MS}.

In this paper we will introduce and investigate an extension of
Skorokhod-Malliavin calculus to random fields generated by an \emph{arbitrary%
} sequence $\Xi =\left( \xi _{1},\xi _{2},...\right) $ of square integrable
and uncorrelated random variables on the probability space $(\Omega ,%
\mathcal{F},\mathbf{P)}$. Let us assume that for every $i,$ $\Pr \left( \xi
_{i}<x\right) =F^{i}\left( x\right) .$ These distribution functions\textit{\
are given but their types are not specified}. Of course, some or all of
these distribution functions could coincide.

The above setting constitute the so called "distribution free" paradigm. As
the title suggests, our task is to develop a version of Malliavin-Skorohod
calculus in the distribution free setting. At the first glance, a rigorous
implementation of the distribution free Skorokhod-Malliavin calculus might
seem to be a long shot, however, it is not completely unexpected. Consider,
for example, the following quotation from P.A. Meyer\ \cite{meyer}:
\textquotedblleft The first and very important point is that, in the
construction of multiple integrals, the Gaussian character of the process
never appears".

The main operators of Skorokhod-Malliavin calculus are Wick product ($%
\diamond )$, Skorokhod integral ($\delta $)\ , and Malliavin derivative ($%
\mathbb{D}$). Skorokhod integral is an anti-derivative of Malliavin
derivative and Wick product is an elementary but useful version of Skorokhod
integral. In the current paper, these operators are defined and investigated
in the distribution free setting under two natural assumptions (see \textbf{%
B1 }and \textbf{B2} below) that hold, practically, for all standard types of
random variables.

The \textit{distribution free} version of Skorokhod-Malliavin calculus,
developed in Sections 2-3, preserves the fundamental properties of the 
\textit{classic} Malliavin- Skorokhod calculus. For example, the It\^{o}%
-Skorokhod isometry

\begin{eqnarray}
\mathbf{E}\left[ |\delta \left( u\right) |^{2}\right] &=&\mathbf{E}%
\left\Vert u\right\Vert ^{2}  \label{ItoIsomntr} \\
&&  \notag \\
\mathbf{E}\left[ |\delta \left( u\right) |^{2}\right] &=&\mathbf{E}%
\left\Vert u\right\Vert ^{2}+\mathbf{E}\left[ (\mathbb{D}u,\mathbb{D}u)%
\right]  \label{AIs}
\end{eqnarray}%
where (\ref{ItoIsomntr}) holds for functions $u$ adapted to the appropriate
filtration and (\ref{AIs}) holds for non-adapted random variables (see
Proposition \ref{pe2}).

In the distribution free setting, it is natural to construct the \textit{%
driving noise} $\mathfrak{\dot{N}}$ as follows:%
\begin{equation*}
\mathfrak{\dot{N}}=\sum_{k}m_{k}\xi _{k}\text{, }\xi _{k}\in \Xi ,
\end{equation*}%
where $\xi _{k\text{ }}$are uncorrelated random variables, $\mathbf{E}\xi
_{k}=0$, $\mathbf{E}\left\vert \xi _{k}\right\vert ^{2}=1$ and $\left\{
m_{k},k\geq 1\right\} $ is an orthogonal basis in some Hilbert space $%
\mathbf{H}=L_{2}\left( U,\mathcal{U},\mu \right) ,$ where $(U,\mathcal{U}%
,\mu )$ is a $\sigma $-finite measure space.

Let $J$ be the set of multiindices $\alpha =\left( \alpha _{1},\alpha
_{2},...\right) $, such that for every $k,$ $\alpha _{k}\in \mathbf{N}_{0}$%
,\ $\mathbf{N}_{0}=\left\{ 0,1,2,...\right\} $ and $\left\vert \alpha
\right\vert =\sum_{k}\alpha _{k}.$ We will construct a complete orthogonal
system $\left\{ \mathfrak{N}_{\alpha },\alpha \in J\right\} $\ in $%
L^{2}\left( \Omega ,\sigma \left( \xi _{k},k\geq 1\right) ,\mathbf{P}\right)
.$ By construction, this\ basis is distribution free. The Cameron-Martin
basis for Gaussian random fields is a particular case of the basis $\left\{ 
\mathfrak{N}_{\alpha },\alpha \in J\right\} .$ The set of deterministic
coefficients $\left\{ X_{\alpha }=\mathbf{E}\left( X\mathfrak{N}_{\alpha
}\right) ,\alpha \in J\right\} $ is often referred to as the \textit{%
propagator} of the random variable/field $X$.

In Section 4, the distribution free Skorokhod-Malliavin "technology" is
applied to the analysis of linear ordinary SDE as well as linear stochastic
parabolic and elliptic SPDEs with additive or multiplicative distribution
free noise. We will study these equations and their relations in adapted and
non-adapted settings.

As an example, let us consider an adapted parabolic SPDE 
\begin{equation}
u(t)=w+\int_{0}^{t}{\mathcal{L}}u(s)ds+\int_{0}^{t}\int_{U}\left[
u(s)G\left( s,\upsilon \right) +f\left( s,\upsilon \right) \right] \diamond 
\mathfrak{N}\left( ds,d\upsilon \right) ,  \label{pfo0}
\end{equation}%
where $\mathcal{L}u=a^{ij}(x)u_{x_{i}x_{j}}+b^{i}(x){u_{x_{i}}.}$ The
propagator associated with (\ref{pfo0}) is given by following low-triangular
system of deterministic PDEs for the coefficients $u_{0}\left( t\right)
=w_{0}$ and 
\begin{equation}
\left\{ 
\begin{array}{l}
\partial _{t}u_{\alpha }(t)={\mathcal{L}}u_{\alpha
}+\sum_{k}\int_{V}m_{k}(u_{\alpha (k)}G+f_{\alpha \left( k\right) })d\pi \\ 
u_{\alpha }(0)=w,%
\end{array}%
\right.  \label{propag0}
\end{equation}%
where $\alpha \left( k\right) =(\alpha _{1},\alpha _{2},...\alpha
_{k-1},\alpha _{k}-1,\alpha _{k+1},...)$ and $\left\vert \alpha \right\vert
>0.$ In contrast to stochastic PDE (\ref{pfo0}), the related propagator is a
deterministic lower-triangular system. Therefore, under quite general
assumptions, system (\ref{propag0}) is solvable. Therefore, one can
construct a distribution free "polynomial chaos" solution of equation (\ref%
{pfo0}) in the form $u=\sum_{\alpha }u_{\alpha }\mathfrak{N}_{\alpha }.$\ 

Note that, due to its lower triangular structure, the propagator (\ref%
{propag0}) can be solved sequentially. The latter bodes well to the
efficiency of numerical implementation of the polynomial chaos solutions.

Similarly, under standard assumptions (e.g. positivity of the operator $A$),
one can construct a polynomial chaos solution of the stationary equation 
\begin{equation}
{\mathbf{A}}u+\sum_{n\geq 1}\mathbf{M}_{n}u\diamond \xi _{n}=f\text{ .}
\label{STBL}
\end{equation}%
The propagator of this equation is given by the system 
\begin{equation}
\begin{array}{c}
{\mathbf{A}}u_{\alpha }=Ef\text{ }\ \text{if }\left\vert \alpha \right\vert
=0 \\ 
{\mathbf{A}}u_{\alpha }+\sum_{n\geq 1}\mathbf{M}_{n}u_{\alpha -\varepsilon
_{n}}=f_{\gamma }\text{ }\ \text{if }\left\vert \alpha \right\vert >0,%
\end{array}
\label{PropEll1}
\end{equation}%
where $\varepsilon _{n}$ is a multiindices with $\left\vert \varepsilon
_{n}\right\vert =1$ and the only non-zero component at at the $n^{th\text{ }%
} $coordinate.

Again, under standard assumptions (including positivity of the operator $A$%
), one can solve the propagator of the stationary equation and reconstruct
the distribution free solution of (\ref{STBL}) in the polynomial chaos form.

Note that, in both settings the triangular feature of the propagator is due
to the linear structure of the underlying stochastic equations.

A much more difficult and nuanced problem is the existence of "distribution
free" solutions for nonlinear stochastic equations driven by arbitrary
noise. In general, one should not expect a "universal" answer, because the
coefficients of expansions of a nonlinear function (e.g. a product) of
random variables depends on the types of this random variables. For more
detail and examples, see \cite{RSA}, \cite{MR2}.

Nevertheless, there exists at least one reasonably broad class of nonlinear
SPDEs that fits into the distribution free paradigm. Specifically, SPDEs
with the so-called Wick-polynomial nonlinearities belong to this class.
Equation 
\begin{equation}
{\mathbf{A}}u-u^{\diamond k}+\sum_{n\geq 1}\mathbf{M}_{n}u\diamond \xi _{n}=f%
\text{ }  \label{ell-3-d}
\end{equation}%
is an important example of a nonlinear equation from this class.

Note that the nonlinear part of equation (\ref{ell-3-d}) is a Wick power. By
this reason, it is easy to see that the propagator of equation (\ref{ell-3-d}%
) is again a linear lower triangular system given by 
\begin{equation}
{\mathbf{A}}u_{\alpha }-\sum_{\kappa ,\beta ,\gamma :\kappa +\beta +\gamma
=\alpha }u_{\kappa }u_{\beta }u_{\gamma }+\sum_{n\geq 1}\mathbf{M}%
_{n}u_{\alpha -\varepsilon _{n}}=f_{\alpha }  \label{STBL2}
\end{equation}%
for all $\alpha \in J.$ Due to its lower-triangular this system is uniquely
solvable (under reasonable assumptions on the operators).

Wick products have also been used for designing unbiased approximations of
stochastic Navier-Stokes equation (see \cite{MR1} and the references
therein). Specifically, in this types of models, the standard nonlinear term 
$\boldsymbol{v\nabla v}$ was approximated by the Wick product $\boldsymbol{%
v\diamond \nabla v}$. The stochastic Navier-Stokes equation with this
correction is unbiased, that is the expectation of the solution coincides
with the related deterministic Navier-Stokes equation.

Note that the first examples of systems with Wick nonlinearities were
introduced and investigated in Euclidean Quantum Field theory (see e.g. \cite%
{Sim}). \ 

The most challenging problem in the analysis of distribution free noise is
development of a distribution free calculus suitable for adapted and
non-adapted "arbitrary" random fields. Not surprisingly, this subject
constitutes the "foundation\textquotedblright\ of the paper. It is addressed
in Section 3. This Section includes construction of \textit{distribution
free }multiple integrals, Skorokhod integral, Malliavin derivatives, and
Wick exponent, as well as Ito-Skorokhod isometry. Linear SDEs and parabolic
SPDEs are discussed in Section 4.02, 4.03. Linear and Wick-nonlinear
elliptic (non-adapted) SPDEs are covered in Section 4.1. The Appendix deals
with Wick products of multiple integrals.

\section{Driving cylindrical random fields}

Let $\left( U,\mathcal{U},\mu \right) $ be a $\sigma $-finite complete
measure space. Let $\mathbf{H}=L_{2}(U,\mathcal{U},\mu )$ and $(\Omega ,%
\mathcal{F},\mathbf{P)}$ be a complete probability space.

\begin{definition}
A continuous linear functional $\mathfrak{N}$ from $\mathbf{H}$ to $%
L_{2}\left( \Omega ,\mathcal{F},\mathbf{P}\right) $ such that $\mathbf{E[}%
\mathfrak{N}\left( f\right) ^{2}]=\left\vert f\right\vert _{2}^{2},f\in 
\mathbf{H}$, is called a \textbf{driving} cylindrical random field on $U.$
\end{definition}

It is assumed that $\mathbf{E}\mathfrak{N}(f)=0,f\in \mathbf{H}$. Clearly, $%
\mathfrak{N}$ is an isometric embedding of $\mathbf{H}$ into $L^{2}\left(
\Omega ,\mathcal{F},\mathbf{P}\right) .$

\begin{remark}
If $f\in \mathbf{H}$ and $\left\{ m_{k}\right\} $ is a complete orthonormal
system in $\mathbf{H}$, then, obviously,%
\begin{equation*}
\mathfrak{N}(f)=\sum_{k}f_{k}\xi _{k}\text{ in }L_{2}\left( \Omega ,\mathcal{%
F},\mathbf{P}\right) ,
\end{equation*}%
where $\xi _{k}=\mathfrak{N}\left( m_{k}\right) $ and $f_{k}=\int_{U}fm_{k}d%
\mu $ (note that $f=\sum_{k}f_{k}m_{k}$ in $\mathbf{H}$). Moreover,%
\begin{equation*}
\mathbf{E[}\mathfrak{N}\left( f\right) \mathfrak{N}\left( g\right) ]=\int
fgd\mu ,~f,g\in \mathbf{H},
\end{equation*}
\end{remark}

The noise (associated to $\mathfrak{N}$) is the formal series%
\begin{equation*}
\mathfrak{\dot{N}}(\upsilon )=\sum_{k}m_{k}\left( \upsilon \right) \xi
_{k},\upsilon \in U.
\end{equation*}%
We write%
\begin{equation*}
\mathfrak{N}(f)=\int_{U}\mathfrak{\dot{N}}\left( \upsilon \right) f(\upsilon
)\mu (d\upsilon )=\int_{U}f(\upsilon )\mathfrak{N}(d\upsilon
)=\sum_{k}f_{k}\xi _{k}.
\end{equation*}

\begin{remark}
\label{r1}If $\left( \xi _{k}\right) $ is an arbitrary sequence of centered
uncorrelated r.v. in $L^{2}\left( \Omega ,\mathbf{P}\right) $, then for any $%
L_{2}(U,\mathcal{U},\mu )$ the map%
\begin{equation*}
\mathfrak{N}\left( f\right) =\sum_{k}f_{k}\xi _{k},~f=\sum_{k}f_{k}m_{k}\in
L^{2}\left( U,\mu \right) ,
\end{equation*}%
is a driving cylindrical random field on $U,$ that is any sequence of
centered uncorrelated r.v. can define a generalized driving random field.
\end{remark}

We shall introduce the following assumptions about $\mathfrak{N}$.

\textbf{B1. }L\emph{et }$\xi _{k}=\mathfrak{N}\left( m_{k}\right) ,k\geq 1$%
\emph{. For each vector r.v. }$\left( \xi _{i_{1}},\ldots ,\xi
_{i_{n}}\right) ,n\geq 1,$\emph{\ the moment generating function }%
\begin{equation*}
M_{i_{1}\ldots i_{n}}(t)=M_{i_{1}\ldots i_{n}}(t_{1},\ldots ,t_{n})=\mathbf{E%
}\exp \left\{ t_{1}\xi _{i_{1}}+\ldots t_{n}\xi _{i_{n}}\right\}
\end{equation*}%
\emph{exists for all }$t=(t_{1},\ldots ,t_{n})$ \emph{in some neighborhood
of }$0\in \mathbf{R}^{n}$\emph{.}

This assumption implies immediately that $\xi _{k}$ have all the moments
(see Theorem 5a, p. 57 in \cite{w}). Let $J$ be the set of all multiindices $%
\alpha =\left( \alpha _{1},\ldots \right) $ such that $\alpha _{k}\in
\left\{ 0,1,2,\ldots \right\} $ and $\left\vert \alpha \right\vert
=\sum_{k}\alpha _{k}<\infty $. For $\alpha =\left( \alpha _{k}\right) \in J$
we denote%
\begin{equation*}
\xi ^{\alpha }=\prod_{k}\xi _{k}^{\alpha _{k}},\xi ^{0}=1.
\end{equation*}

The following assumption is needed as well.

\textbf{B2. }\emph{Assume we are given an orthogonalization }$\left\{ \tilde{%
K}_{\alpha },\alpha \in J\right\} $\emph{\ of the system }$O=\left\{ \xi
^{\alpha },\alpha \in J\right\} $\emph{\ such that for each }$n,$\emph{\ }$%
\left\{ \tilde{K}_{p},|p|\leq n\right\} $\emph{\ spans the same linear
subspace }$H_{n}$\emph{\ as }$\left\{ \xi ^{p},|p|\leq n\right\} $\emph{\
and for }$\left\vert \alpha \right\vert =n+1,$\emph{\ }%
\begin{equation*}
\tilde{K}_{\alpha }=\xi ^{\alpha }-\text{projection}_{H_{n}}\xi ^{\alpha }
\end{equation*}

Set $\mathfrak{N}_{\alpha }=c_{\alpha }\tilde{K}_{\alpha }$ so that $\mathbf{%
E}\left[ \mathfrak{N}_{\alpha }^{2}\right] =\alpha !$. The result is the
complete orthogonal system $\left\{ \mathfrak{N}_{\alpha },\alpha \in
J\right\} .$Obviously, $\mathfrak{N}_{0}=1,$and for $p\in J,p=\varepsilon
_{k}$ ($\varepsilon _{k}\in J$ and has all components zeros except $1$ as
the $k$th component), $\mathfrak{N}_{p}=\mathfrak{N}_{\varepsilon _{k}}=\xi
_{k}.$

\begin{remark}
\label{r2}a) If every $\xi _{k}=\mathfrak{N}\left( m_{k}\right) $ is
bounded, then \textbf{B1 }obviously holds.

b) The Hilbert space $\mathbf{H}$ can be finite dimensional. In particular,
if $U=\left\{ 1\right\} $ and $\mu $ is the Dirac measure $\delta _{1},$
then $\mathbf{H}$ is one-dimensional, $m_{1}=1$. If $\xi =\mathfrak{N}(1)$,
then the $\mathfrak{N}$-noise coincides with the r.v. $\mathfrak{\dot{N}}%
=\xi $.
\end{remark}

The following statement is almost obvious.

\begin{proposition}
\label{p1}Assume \textbf{B1, B2 }hold. Let $\mathcal{F}^{0}=\sigma \left(
\xi _{k},k\geq 1\right) =\sigma \left( \mathfrak{N}(f),f\in \mathbf{H}%
\right) .$ Then $\left\{ \mathfrak{N}_{\alpha }=\xi ^{\diamond \alpha
},\alpha \in J\right\} $ is a complete orthogonal system of $L_{2}\left(
\Omega ,\mathcal{F}^{0},\mathbf{P}\right) :$ for each $\eta \in L_{2}\left(
\Omega ,\mathcal{F}^{0},\mathbf{P}\right) ,$%
\begin{equation*}
\eta =\sum_{\alpha }\eta _{\alpha }\mathfrak{N}_{\alpha },
\end{equation*}%
where%
\begin{equation*}
\eta _{\alpha }=\frac{\mathbf{E}\left[ \eta \mathfrak{N}_{\alpha }\right] }{%
\alpha !}.
\end{equation*}%
Note that $\sum_{\alpha }\eta _{\alpha }^{2}\alpha !=\mathbf{E}\left[ \eta
^{2}\right] <\infty .$
\end{proposition}

\begin{proof}
By Lemma \ref{ale1}, \textbf{B1 }implies that any $\eta \in L^{2}\left( 
\mathcal{F}^{0},\mathbf{P}\right) $ can be approximated by a sequence of
polynomials in $\xi ^{\alpha },\alpha \in J$. Therefore the
orthogonalization described in \textbf{B2 }defines a complete orthogonal
system and $\left\{ \mathfrak{N}_{\alpha }/\sqrt{\alpha !},\alpha \in
J\right\} $ is a CONS in $L^{2}\left( \mathcal{F}^{0},\mathbf{P}\right) $.
\end{proof}

\subsection{Examples}

\subsubsection{Driving cylindrical random fields and processes generated by
independent r.v.'s}

\begin{example}
\label{ex1}In the case of a single r.v. $\xi $ as in Remark \ref{r2}, $%
J=\left\{ 0,1,2,\ldots \right\} ,$ the complete orthogonal system $\left\{ 
\mathfrak{N}_{n},n\geq 0\right\} $ of $L_{2}\left( \mathcal{\sigma }\left(
\xi \right) ,\mathbf{P}\right) $ must coincide with the one obtained by
Gram-Scmidt orthogonalization procedure. We set $\mathfrak{N}_{0}=1,%
\mathfrak{N}_{1}=\xi $. If $H_{n}$ be the subspace generated by $\left\{ 
\mathfrak{N}_{l},l\leq n\right\} ,$ then we take%
\begin{equation*}
\tilde{K}_{n+1}=\xi ^{n+1}-\text{projection}_{H_{n}}\left( \xi ^{n+1}\right)
,
\end{equation*}%
and set $\mathfrak{N}_{n+1}=c_{n+1}\tilde{K}_{n+1}$ so that $\mathbf{E}\left[
\mathfrak{N}_{n+1}^{2}\right] =(n+1)!$.
\end{example}

\begin{example}
\label{ex2}Let $\xi _{k}$ be a sequence of independent centered r.v. whose
moment generating function exists in a neighborhood of zero. Assume$\ 
\mathbf{E}\left( \xi _{k}^{2}\right) =1$.

As already discussed in Remark \ref{r1}, for any $L_{2}(U,\mathcal{U},\mu )$
the map%
\begin{equation*}
\mathfrak{N}\left( f\right) =\sum_{k}f_{k}\xi _{k},~f=\sum_{k}f_{k}m_{k}\in
L^{2}\left( U,\mu \right) ,
\end{equation*}%
is a driving cylindrical random field on $U$. Obviously \textbf{B1 }holds.
For every $k$ we apply orthogonalization procedure of Example \ref{ex1} and
construct $\mathfrak{N}_{l}^{k}$ with $\mathbf{E}\left[ \left( \mathfrak{N}%
_{l}^{k}\right) ^{2}\right] =l!.$ For a multiindex $\alpha \in J$, we set$,$%
\begin{equation}
\mathfrak{N}_{\alpha }=\prod_{k}\mathfrak{N}_{\alpha _{k}}^{k}.  \label{fo3}
\end{equation}%
Note that $\mathbf{E}\left[ \mathfrak{N}_{\alpha }^{2}\right] =\alpha !$ For
any $\alpha \in J$ with $\left\vert \alpha \right\vert =n+1$, $\mathfrak{N}%
_{\alpha }$ is orthogonal to $H_{n}$ and of the form $\xi ^{\alpha
}-l_{\alpha }$, where $l_{\alpha }$ is a linear combination of vectors in $%
H_{n}$, i.e. $l_{\alpha }$ is the orthogonal projection of $\mathfrak{\xi }%
^{\alpha }$ on $H_{n}$. Therefore \textbf{B1 }is satisfied as well and $%
\left\{ \mathfrak{N}_{\alpha },\alpha \in J\right\} $ is a CONS in $%
L^{2}\left( \Omega ,\sigma \left( \xi _{k},k\geq 1\right) ,\mathbf{P}\right) 
$.
\end{example}

\begin{example}
\label{ex3}Let $\xi _{k}$ be a sequence independent centered r.v. that have
all the moments and $\mathbf{E}\left( \xi _{k}^{2}\right) =1$. Let $(U,%
\mathcal{U},\mu )$ be a $\sigma $-finite measure space. Let $\left\{
m_{k},k\geq 1\right\} $ be a CONS in $\mathbf{H}=L_{2}\left( U,\mathcal{U}%
,\mu \right) $. We can define $\mathfrak{N}:L_{2}\left( U,\mathcal{U},\mu
\right) \rightarrow L_{2}\left( \Omega ,\mathbf{P}\right) $ as 
\begin{equation}
\mathfrak{N}\left( f\right) =\sum_{k}f_{k}\xi _{k},f=\sum_{k}f_{k}m_{k}\in 
\mathbf{H}\text{,}  \label{fo1}
\end{equation}%
and the complete orthogonal system $\left\{ \mathfrak{N}_{\alpha },\alpha
\in J\right\} $ as in Example \ref{ex2}. The $\mathfrak{N}$-noise is%
\begin{equation*}
\mathfrak{\dot{N}}(x)=\sum_{k}m_{k}(x)\xi _{k}.
\end{equation*}

In particular, if $U=L^{2}\left( [0,T]\right) $ and $\left\{ m_{k},k\geq
1\right\} $ is a CONS on $L^{2}\left( [0,T]\right) ,$ then we can regard $%
\mathfrak{N}$ in (\ref{fo1}) as a stochastic process%
\begin{equation}
\mathfrak{N}_{t}=\mathfrak{N}\left( \chi _{\left[ 0,t\right] }\right)
=\sum_{k}\int_{0}^{t}m_{k}\left( s\right) ds\xi _{k},0\leq t\leq T.
\label{fo2}
\end{equation}%
It has uncorrelated increments, $\mathbf{E}\left( \mathfrak{N}%
_{t}^{2}\right) =t,\mathbf{E}\left( \mathfrak{N}_{t}\mathfrak{N}_{s}\right)
=t\wedge s,\mathbf{E}\left( \mathfrak{N}_{t}\right) =0,t\geq 0$. For any
continuous deterministic $f(t)$ on $[0,T]:$%
\begin{equation*}
\mathfrak{N}\left( f\right) =\lim_{n}\sum_{i=1}^{n}f\left( t_{i}\right) %
\left[ \mathfrak{N}_{t_{i+1}}-\mathfrak{N}_{t_{i}}\right] \text{ in }%
L^{2}\left( \Omega ,\mathbf{P}\right) ,
\end{equation*}%
where $t_{i}=t_{i}^{n},0\leq i\leq n,$ is a partition of $\left[ 0,T\right] $
into $n$ disjoint subintervals whose maximal size converges to zero as $%
n\rightarrow \infty $.
\end{example}

Some specific examples related to Example \ref{ex3}.

\begin{itemize}
\item For a standard normal $\xi \sim N\left( 0,1\right) ,$ the sequence $%
\mathfrak{N}_{n}$ in Example \ref{ex1} are Hermite polynomials: $\mathfrak{N}%
_{n}=\frac{d^{n}}{dz^{n}}p\left( z\right) \}|_{z=0}$ with%
\begin{equation*}
p\left( z\right) =\exp \left\{ z\xi -\frac{1}{2}z^{2}\right\} ,z\in \mathbf{%
R.}
\end{equation*}%
With a sequence $\xi _{k},k\geq 1,$ of independent standard normals,%
\begin{equation*}
W_{t}=\sum_{k}\int_{0}^{t}m_{k}(s)ds\xi _{k},0\leq t\leq T,
\end{equation*}%
is a standard Wiener process (see \cite{in}).

\item Let $\xi _{k}$ be i.i.d. such that $\mathbf{P}\left( \xi _{k}=1\right)
=\mathbf{P}\left( \xi _{k}=-1\right) =1/2.$ Then the driving process $%
\mathfrak{N}_{t}$ in (\ref{fo2}) is not Gaussian: for each $n>1,u\in \mathbf{%
R},$%
\begin{equation*}
\mathbf{E}\exp \left\{ \iota u\sum_{k=1}^{n}\int_{0}^{t}m_{k}(s)ds\xi
_{k}\right\} =\prod_{k=1}^{n}\cos \left\{ u\int_{0}^{t}m_{k}(s)ds\right\} .
\end{equation*}
It is straightforward to find ($\xi _{k}^{2}=1$ a.s.) that for%
\begin{equation*}
\mathbf{E}\left[ \left( W_{t}^{n}-W_{s}^{n}\right) ^{4}\right] \leq
C|t-s|^{2},~s,t\in \lbrack 0,T],
\end{equation*}%
where $W_{t}^{n}=\sum_{k=1}^{n}\int_{0}^{t}m_{k}\left( s\right) ds\xi
_{k},0\leq t\leq T$. Therefore, by the same arguments as in the Gaussian
case (see \cite{in}), $W_{t}$ has a continuous in $t$ modification.

\item Let $\xi _{k}$ be uniform i.i.d. on $[-\sqrt{3},\sqrt{3}]$. Again, we
have in (\ref{fo2}) a non-Gaussian driving process with continuous paths:%
\begin{equation*}
\mathbf{E}\exp \left\{ \iota u\sum_{k=1}^{n}\int_{0}^{t}m_{k}(s)ds\xi
_{k}\right\} =\prod_{k=1}^{n}\frac{\sin \left(
u\int_{0}^{t}m_{k}(s)ds\right) }{\left( u\int_{0}^{t}m_{k}(s)ds\right) }.
\end{equation*}
The orthogonal basis consists of Legendre polynomials in this case. Let us
consider the set of Legendre polynomials $L_{k}$ on $\left[ -1,1\right] $
defined by the Rodrigues formula 
\begin{equation}
L_{n}(\eta )=\frac{(-1)^{n}}{2^{n}n!}\frac{d^{n}}{d\eta ^{n}}\left[ (1-\eta
^{2})^{n}\right] .  \label{legendre1}
\end{equation}%
For any $n,m$, 
\begin{equation}
\frac{1}{2\sqrt{3}}\int_{-\sqrt{3}}^{\sqrt{3}}L_{n}(\frac{x}{\sqrt{3}})L_{m}(%
\frac{x}{\sqrt{3}})dx=\frac{1}{2}\int_{-1}^{1}L_{n}(\eta )L_{m}(\eta )d\eta =%
\frac{1}{(2n+1)}\delta _{nm}\,.
\end{equation}%
Let $\tilde{B}_{n}=L_{n}\left( x/\sqrt{3}\right) $. To make it "standard" ($%
\mathbf{E}\left( B_{n}^{2}\right) =n!$)$,$we set%
\begin{equation*}
B_{0}=1,B_{n}=\sqrt{n!(2n+1)}\tilde{B}_{n}.
\end{equation*}%
In this case $\mathfrak{N}_{k}^{n}=\xi _{k}^{\diamond n}=B_{n}\left( \xi
_{k}\right) ,k\geq 1,n\geq 0$ and we get the whole system by (\ref{fo3}).
\end{itemize}

A sequence of independent driving random fields can be used to construct a
new one.

\subsubsection{Poisson random fields}

Let $N$ be Poisson random measure on $U$ with $\mu $ as its Levy measure and 
$\tilde{N}=N-\mu $. We have isometry:%
\begin{equation*}
\mathbf{E}\left[ \tilde{N}(\varphi )^{2}\right] =\int_{U}\varphi ^{2}d\mu
,\varphi \in L_{2}\left( U,\mu \right) .
\end{equation*}

Let $m_{k}(x)$ be a CONS in $L_{2}(U,\nu )$ such that all $m_{k}$ and $%
\int_{U}|m_{k}(\upsilon )|~d\mu $ are bounded. In general, $\xi _{k}=\tilde{N%
}\left( m_{k}\right) $ are not independent.

Let $\mathcal{Z}$ be the set of all real-valued sequences $z=\left(
z_{k}\right) $ such that only the finite number of $z_{k}$ is not zero. For $%
z\in \mathcal{Z}$ set, $m=m_{z}=\sum_{k}z_{k}m_{k}$. Under the assumptions
above, the moment generating function%
\begin{eqnarray*}
M(z) &=&\mathbf{E}\exp \left\{ \tilde{N}(m_{z})\right\} \\
&=&\exp \left\{ \int_{U}\left[ e^{m_{z}\left( \upsilon \right)
}-1-m_{z}\left( \upsilon \right) \right] \mu \left( d\upsilon \right)
\right\}
\end{eqnarray*}%
exists and Assumption \textbf{B1} is satisfied.

\textbf{Charlier polynomials}

For small $z,$ let 
\begin{equation*}
p(z)=\exp \left\{ \int_{U}\ln [1+m_{z}(\upsilon )]N(d\upsilon
)-\int_{U}m_{z}(\upsilon )\mu (d\upsilon )\right\} .
\end{equation*}

For $\alpha \in J,$ we define Charlier polynomials as%
\begin{equation*}
C_{\alpha }=\frac{\partial ^{\alpha }}{\partial z^{\alpha }}%
p(z)|_{z=0},\alpha \in J.
\end{equation*}%
For example, if $|\alpha |=1,\alpha =\varepsilon _{k}$, then $C_{\alpha }=%
\tilde{N}\left( m_{k}\right) ;$ If $|\alpha |=2,\alpha =\varepsilon
_{k_{1}}+\varepsilon _{k_{2}}$, then%
\begin{eqnarray*}
C_{\alpha } &=&N\left( m_{k_{1}}\right) N\left( m_{k_{2}}\right)
-N(m_{k_{1}}m_{k_{2}})-N(m_{k_{1}})\mu (m_{k_{2}}) \\
&&-N(m_{k_{2}})\mu (m_{k1})+\mu (m_{k_{1}})\mu (m_{k_{2}}),
\end{eqnarray*}%
where $N(\phi )=\int_{U}\phi (x)~N(dx),\mu (\phi )=\int_{U}\phi (x)\mu (dx)$%
. Recall $\varepsilon _{k}=\left( \alpha _{i}\right) \in J$ and has all
components zero except $\alpha _{k}=1.$

It is a standard fact that Charlier polynomials are orthogonal: $\mathbf{E}%
C_{\alpha }C_{\alpha ^{\prime }}=\alpha !$ if $\alpha =\alpha ^{\prime }$
and zero otherwise. In general and contrary to the Hermite polynomials for
the Gaussian random variables, they are not polynomials of Poisson random
variables.

We define driving Poisson noise on $U$ as $\mathfrak{N}\left( f\right) =%
\tilde{N}\left( f\right) ,f\in L^{2}\left( U,d\mu \right) $. From our
general considerations above, it satisfies \textbf{B1, B2 }with $\mathfrak{N}%
_{\alpha }=C_{\alpha },\alpha \in J.$

\subsection{Generalized random variables and processes}

{Let }$E$ be a topological vector space and {\ }%
\begin{equation*}
\mathcal{D(}E\mathcal{)=}\left\{ v=\sum_{\alpha }v_{\alpha }\mathfrak{N}%
_{\alpha }:\text{ }v_{\alpha }\in E\text{ and only finite number of }%
v_{\alpha }\text{ are not zero}\right\} .
\end{equation*}

\begin{definition}
{A generalized $\mathcal{D}$-random variable with values in }$E${\ }with
Borel $\sigma $-algebra{\ is a formal series }$u=\sum_{\alpha }u_{\alpha
}\xi _{\alpha }${, where }$u_{\alpha }\in E.$
\end{definition}

{Denote the vector space of all generalized }$\mathcal{D}${-random variables
by }$\mathcal{D}^{\prime }=\mathcal{D}^{\prime }(E)${. }If $E=\mathbf{R}$ we
write simply $\mathcal{D},\mathcal{D}^{\prime }$. {The elements of }$%
\mathcal{D}${\ are the test random variables for }$\mathcal{D}^{\prime }.${\
We define the action of a generalized random variable }$u\in \mathcal{D}%
^{\prime }\left( E\right) ${\ on the test random variable }$v\in \mathcal{D}$%
{\ by }$\left\langle u,v\right\rangle =\sum_{\alpha }v_{\alpha }u_{\alpha }.$
If $Y$ is Hilbert and $u\in \mathcal{D}^{\prime }(Y),v\in \mathcal{D}\left(
E\right) $, then $\left\langle u,v\right\rangle =\sum_{\alpha }(v_{\alpha
},u_{\alpha })_{Y}.$

{For a sequence }$u^{n}\in \mathcal{D}^{\prime }${\ and }$u\in \mathcal{D}%
^{\prime }${, we say that }$u^{n}\rightarrow u${, if for every }$v\in 
\mathcal{D},$ $\left\langle u,v^{n}\right\rangle \rightarrow \left\langle
u,v\right\rangle .$ {This implies that }$u^{n}=\sum_{\alpha }u_{\alpha }^{n}%
\mathfrak{N}_{\alpha }\rightarrow u=\sum_{\alpha }u_{\alpha }\mathfrak{N}%
_{\alpha }${\ if and only if }$u_{\alpha }^{n}\rightarrow u_{\alpha }${\ as }%
$n\rightarrow \infty ${\ for all }$\alpha ${.}

\begin{remark}
\label{rem11}If $u=\sum_{\alpha }u_{\alpha }\mathfrak{N}_{\alpha }\in 
\mathcal{D}^{\prime }(E),~F$ is a vector space and $f:E\rightarrow F$ is a
linear map, then we define%
\begin{equation*}
f(u)=\sum_{\alpha }f(u_{\alpha })\mathfrak{N}_{\alpha }\in \mathcal{D}%
^{\prime }(F).
\end{equation*}
\end{remark}

\begin{definition}
{An\ }$E${-valued generalized }$\mathcal{D}${- field on a measurable space }$%
(B,\mathcal{B})$ {is a }$\mathcal{D}^{\prime }(E)${-valued function on }$B${%
\ such that for each }$x\in B,$%
\begin{equation*}
u(x)=\sum_{\alpha }u_{\alpha }(x)\mathfrak{N}_{\alpha }\in \mathcal{D}%
^{\prime }(E),
\end{equation*}%
where $u_{\alpha }(x)$ are deterministic measurable $E$-valued functions on $%
B$.
\end{definition}

We denote the linear space of all such fields by{\ }$\mathcal{D}^{\prime
}(B;E).$ If $E$ is a topological vector space and a generalized $\mathcal{D}$%
{-field} $u(x)$ is continuous on $B$ we write $u\in $ $C\mathcal{D}^{\prime
}(B;E)$ (note that $u(x)$ is continuous if and only if all coefficient
functions $u_{\alpha }$ on $B$ are continuous$)$. In particular, if $B=\left[
0,T\right] $, we say $u\left( t\right) $ is a generalized $\mathcal{D}$%
-process. {If there is no room for confusion, we will often say }$\mathcal{D}
${-process (}$\mathcal{D}${-random variable) instead of generalized }$%
\mathcal{D}${-process (generalized }$\mathcal{D}${-random variable).}

{If }$\left( B,\mathcal{B},\kappa \right) $ is a measure space and $E${\ is
a normed vector space, we denote}%
\begin{align*}
& L_{2}(\mathcal{D}^{\prime }\mathcal{(}B;E),\kappa ) \\
& =\{u(x)=\sum_{\alpha }u_{\alpha }(x)\mathfrak{N}_{\alpha }\in \mathcal{D}%
^{\prime }\mathcal{(}B;E):\int_{B}|u_{\alpha }(x)|_{E}^{2}d\kappa <\infty 
\text{,~}\alpha \in J\}.
\end{align*}

{For }$u(t)=\sum_{\alpha }u_{\alpha }(t)\mathfrak{N}_{\alpha }\in L_{1}(%
\mathcal{D}^{\prime }([0,T],E))${\ we define }$\int_{0}^{t}u(s)ds,0\leq
t\leq T,${\ in }$\mathcal{D}^{\prime }([0,T];E)${\ by}%
\begin{equation*}
\int_{0}^{t}u(s)ds=\sum_{\alpha }\left( \int_{0}^{t}u_{\alpha }(s)ds\right) 
\mathfrak{N}_{\alpha },0\leq t\leq T.
\end{equation*}

{If }$u(t)=\sum_{\alpha }u_{\alpha }(t)\mathfrak{N}_{\alpha }\in \mathcal{D}%
^{\prime }\mathcal{(}[0,T];E),$ then $u(t)$ is differentiable in $t$ if and
only if $u_{\alpha }(t)${\ are differentiable in }$t\,$. In that case{, }%
\begin{equation*}
\frac{d}{dt}u(t)=\dot{u}(t)=\sum_{\alpha }\dot{u}_{\alpha }(t)\mathfrak{N}%
_{\alpha }\in \mathcal{D}^{\prime }\mathcal{(}[0,T],E).
\end{equation*}

\section{\label{MSC}\textbf{Distribution Free} Skorokhod-Maliavin\textbf{\
Calculus}}

Let $\mathbf{H}^{\otimes n}=L_{2}\left( U^{n},\mathcal{U}^{\otimes n},\mu
_{n}\right) $, where $U^{n}=U\times \ldots \times U$ n-times, $\mu _{n}=\mu
^{\otimes n}.$ Let $\mathbf{H}^{\hat{\otimes}n}$ be the symmetric part of $%
\mathbf{H}^{\otimes n}:$ it is the set of all symmetric $\mu _{n}$-square
integrable functions om \thinspace $U^{n}.$ For $\alpha \in J$ with $%
\left\vert \alpha \right\vert =n$ we define its characteristic set $%
K_{\alpha }=\{k_{1},\ldots ,k_{n}\}$ with each $k$ represented in it by $%
\alpha _{k}$ copies. Let $\mathcal{G}_{n}$ \ be the group of permutations of 
$\left\{ 1,\ldots ,n\right\} $ and%
\begin{equation*}
E_{\alpha }=\sum_{\sigma \in \mathcal{G}_{n}}m_{k_{\sigma (1)}}\otimes
\ldots \otimes m_{k_{\sigma (n)}}.
\end{equation*}%
Then%
\begin{equation*}
e_{\alpha }=\frac{E_{\alpha }}{\sqrt{\alpha !\left\vert \alpha \right\vert !}%
},\alpha \in J,\left\vert \alpha \right\vert =n,
\end{equation*}%
is a CONS of the symmetric part of $\mathbf{H}^{\hat{\otimes}n}$. If $%
\,|p|=1 $ with $k$th component non zero, then $E_{p}=e_{p}=m_{k}$.

\subsection{\textbf{Multiple integrals, Wick product and Skorokhod integral }%
}

\textbf{Multiple integrals}

Now we construct multiple integrals with respect to $\mathfrak{N}$ on the
symmetric part $\mathbf{H}^{\hat{\otimes}n}$ ($\mathbf{H}^{\hat{\otimes}0}=%
\mathbf{R}$). Set for $\left\vert \alpha \right\vert =n\geq 1,$%
\begin{equation*}
I_{n}\left( E_{\alpha }\right) =n!\mathfrak{N}_{\alpha }.
\end{equation*}%
Let $Y$ be a Hilbert space and denote $\mathbf{H}^{\hat{\otimes}n}(Y)$ the
space of all $Y$-valued symmetric functions $v=\sum_{\left\vert \alpha
\right\vert =n}v_{\alpha }E_{\alpha }$ on $U^{n}$ such that%
\begin{equation*}
\left\vert v\right\vert _{\mathbf{H}^{\hat{\otimes}n}(Y)}^{2}=\int_{U^{n}}%
\left\vert v\left( r\right) \right\vert _{Y}^{2}d\mu _{n}=\sum_{\alpha
}\left\vert v_{\alpha }\right\vert _{Y}^{2}\left\vert \alpha \right\vert
!\alpha !<\infty .
\end{equation*}%
For $v=\sum_{\left\vert \alpha \right\vert =n}v_{\alpha }E_{\alpha }\in 
\mathbf{H}^{\hat{\otimes}n}(Y),$ we define its $n$-tuple integral as%
\begin{equation*}
I_{n}\left( v\right) =\sum_{\left\vert \alpha \right\vert =n}v_{\alpha
}I_{n}\left( E_{\alpha }\right) =\sum_{\left\vert \alpha \right\vert
=n}v_{\alpha }n!\mathfrak{N}_{\alpha },
\end{equation*}%
and $I_{0}\left( c\right) =c,c\in \mathbf{R}$. Note that 
\begin{equation*}
\int_{U^{n}}\left( \sum_{\left\vert \alpha \right\vert =n}v_{\alpha
}E_{\alpha }\right) ^{2}d\mu _{n}=\sum_{\left\vert \alpha \right\vert
=n}v_{\alpha }^{2}n!\alpha !
\end{equation*}%
and 
\begin{equation*}
\mathbf{E}\left[ |I_{n}(v)|_{Y}^{2}\right] =n!^{2}\sum_{\left\vert \alpha
\right\vert =n}|v_{\alpha }|_{Y}^{2}\alpha !=n!\left\vert v\right\vert _{%
\mathbf{H}^{\hat{\otimes}n}(Y)}^{2}.
\end{equation*}

Let $\mathcal{S(}Y\mathcal{)}$ be the space of all finite linear
combinations $\sum_{k}\frac{I_{k}\left( F_{k}\right) }{k!}$ with $F_{k}\in 
\mathbf{H}^{\hat{\otimes}k}(Y)$.

\begin{definition}
{A generalized }$\mathcal{S}${-random variable is a formal sum }%
\begin{equation*}
u=\sum_{k}\frac{I_{k}\left( F_{k}\right) }{k!}\text{ with }F_{k}\in \mathbf{H%
}^{\hat{\otimes}k}(Y).
\end{equation*}
We denote the set of all generalized $\mathcal{S}$-random variables by $%
\mathcal{S}^{\prime }\left( Y\right) $.
\end{definition}

The action of $u\in $ $\mathcal{S}^{\prime }\left( Y\right) $ on $v\in $ $%
\mathcal{S}\left( Y\right) $ is defined as 
\begin{equation*}
\left\langle u,v\right\rangle =\sum_{k}\int_{U^{k}}(u_{k},v_{k})_{Y}d\mu ,
\end{equation*}%
where $u=\sum_{k}I_{k}\left( u_{k}\right) /k!,v=\sum_{k}I_{k}\left(
v_{k}\right) /k!$ with $u_{k},v_{k}\in \mathbf{H}^{\hat{\otimes}k}(Y).$

\begin{definition}
{A generalized }$\mathcal{S}${- field on a measurable space }$(B,\mathcal{B}%
) $ {is a }$\mathcal{S}^{\prime }(Y)${-valued function on }$B${\ such that
for each }$x\in B,$%
\begin{equation*}
u(x)=\sum_{n}\frac{I_{n}(F_{n}(x))}{n!}\in \mathcal{S}^{\prime }(Y),
\end{equation*}%
where $F_{n}(x)=F_{n}(x;\upsilon _{1},\ldots ,\upsilon _{n})$ are
deterministic measurable $\mathbf{H}^{\hat{\otimes}n}(Y)$-valued functions
on $B$.
\end{definition}

We denote the linear space of all such fields by{\ }$\mathcal{S}^{\prime
}(B;E).$ If a generalized $\mathcal{S}${-field} $u(x)$ is continuous on $B$
we write $u\in $ $C\mathcal{S}^{\prime }(B;E)$ (note that $u(x)$ is
continuous if and only if all coefficient functions $F_{n}$ on $B$ (as $%
\mathbf{H}^{\hat{\otimes}k}(Y)$-valued functions) are continuous$)$. In
particular, if $B=\left[ 0,T\right] $, we say $u\left( t\right) $ is a
generalized $\mathcal{S}$-process.

\begin{remark}
\label{re2}Obviously, $\mathcal{D}\subseteq \mathcal{S}\left( \mathbf{R}%
\right) $ and $\mathcal{S}^{\prime }\left( Y\right) \subseteq \mathcal{D}%
^{\prime }(Y)$. In fact,%
\begin{equation*}
\mathcal{S}^{\prime }(Y)=\left\{ u=\sum_{\alpha }u_{\alpha }\mathfrak{N}%
_{\alpha }\in \mathcal{D}^{\prime }\left( Y\right) :\sum_{\left\vert \alpha
\right\vert =n}\left\vert u_{\alpha }\right\vert _{Y}^{2}\alpha !<\infty 
\text{ }\forall n\geq 1\right\} .
\end{equation*}

Indeed, if $u=\sum_{\alpha }u_{\alpha }\mathfrak{N}_{\alpha }$ with $%
\sum_{\left\vert \alpha \right\vert =n}\left\vert u_{\alpha }\right\vert
_{Y}^{2}\alpha !<\infty ,n\geq 1$, then%
\begin{equation*}
u_{n}=\sum_{\left\vert \alpha \right\vert =n}u_{\alpha }E_{\alpha }\in 
\mathbf{H}^{\hat{\otimes}n}(Y)
\end{equation*}%
and%
\begin{equation*}
u=\sum_{\alpha }u_{\alpha }\mathfrak{N}_{\alpha }=\sum_{n=0}^{\infty
}\sum_{\left\vert \alpha \right\vert =n}u_{\alpha }\mathfrak{N}_{\alpha
}=\sum_{n=0}^{\infty }\frac{I_{n}\left( u_{n}\right) }{n!}\in \mathcal{S}%
^{\prime }(Y).
\end{equation*}
\end{remark}

Note that for $\left\vert \alpha \right\vert =n,$%
\begin{equation*}
u_{\alpha }=\frac{1}{\alpha !n!}\int_{U^{n}}u_{n}(\upsilon )E_{\alpha
}\left( \upsilon \right) d\mu _{n},n\geq 1.
\end{equation*}

For $n\geq 0$, let $\mathcal{E}\mathbf{H}^{\hat{\otimes}n}$ be the space of
all finite linear combinations of $E_{\alpha },\left\vert \alpha \right\vert
=n$. The following statement provides some insight about the transition from 
$n$-tuple integral to an integral on $U^{n+1}.$

\begin{proposition}
\label{fp3}Let $f\in \mathcal{E}\mathbf{H}^{\hat{\otimes}n},g\in \mathcal{E}%
\mathbf{H}$, $f\otimes g=f\left( z\right) g\left( \upsilon \right) ,z\in
U^{n},\upsilon \in U,$ and let $\widetilde{f\otimes g}$ be the standard
symmetrization of $f\otimes g$. Then%
\begin{equation*}
I_{n+1}\left( \widetilde{f\otimes g}\right) =I_{n}\left( f\right)
I_{1}\left( g\right) -\text{projection}_{H_{n}}\left[ I_{n}\left( f\right)
I_{1}\left( g\right) \right] .
\end{equation*}
\end{proposition}

\begin{proof}
Let 
\begin{equation*}
f=\sum_{\left\vert p\right\vert =n}f_{p}E_{p},g=\sum_{\left\vert
p\right\vert =1}g_{p}E_{p}.
\end{equation*}%
Then%
\begin{eqnarray*}
fg &=&\sum_{p,p^{\prime }}f_{p}g_{p^{\prime }}E_{p}E_{p^{\prime }}, \\
\widetilde{f\otimes g} &=&\sum_{p,p^{\prime }}f_{p}g_{p^{\prime }}\widetilde{%
E_{p}E_{p^{\prime }}}=\frac{1}{n+1}\sum_{p,p^{\prime }}f_{p}g_{p^{\prime
}}E_{p+p^{\prime }}
\end{eqnarray*}%
and%
\begin{eqnarray*}
I_{n+1}\left( \widetilde{f\otimes g}\right) &=&\frac{1}{n+1}%
\sum_{p,p^{\prime }}f_{p}g_{p^{\prime }}I_{n+1}(E_{p+p^{\prime
}})=n!\sum_{p,p^{\prime }}f_{p}g_{p^{\prime }}\mathfrak{N}_{p+p^{\prime }},
\\
I_{n}\left( f\right) &=&n!\sum_{p}f_{p}\mathfrak{N}_{p},I_{1}\left( g\right)
=\sum_{p^{\prime }}g_{p^{\prime }}\mathfrak{N}_{p^{\prime }}
\end{eqnarray*}%
Since%
\begin{equation*}
\mathfrak{N}_{p+p^{\prime }}=\mathfrak{N}_{p}\mathfrak{N}_{p^{\prime }}-%
\text{projection}_{H_{n}}\left[ \mathfrak{N}_{p}\mathfrak{N}_{p^{\prime }}%
\right] ,
\end{equation*}%
it follows that%
\begin{eqnarray*}
I_{n+1}\left( \widetilde{f\otimes g}\right) &=&n!\sum_{p,p^{\prime
}}f_{p}g_{p^{\prime }}[\mathfrak{N}_{p}\mathfrak{N}_{p^{\prime }}-\text{%
projection}_{H_{n}}\left( \mathfrak{N}_{p}\mathfrak{N}_{p^{\prime }}\right) ]
\\
&=&I_{n}(f)I_{1}\left( g\right) -\text{projection}_{H_{n}}\left[
I_{n}(f)I_{1}\left( g\right) \right] .
\end{eqnarray*}
\end{proof}

\begin{remark}
If $U=[0,T],d\mu =dt$ and $m_{1}=\chi _{(a,b)}$ (an interval of unit
length), then according to Proposition \ref{fp3}, the "measure" of the square%
\begin{equation*}
I_{2}\left( \chi _{(a,b)}^{\otimes 2}\right) =I_{1}\left( \chi
_{(a,b)}\right) ^{2}-\text{projection}_{H_{1}}\left[ I_{1}\left( \chi
_{(a,b)}\right) ^{2}\right] .
\end{equation*}
\end{remark}

\textbf{Wick product and Skorokhod integral. }We define Wick product%
\begin{equation*}
\mathfrak{N}_{\alpha }\diamond \mathfrak{N}_{\beta }=\mathfrak{N}_{\alpha
+\beta },1\diamond \mathfrak{N}_{\alpha }=\mathfrak{N}_{\alpha },\alpha
,\beta \in J.
\end{equation*}%
For $u=\sum_{\alpha }u_{\alpha }\mathfrak{N}_{\alpha },v=\sum_{\alpha
}v_{\alpha }\mathfrak{N}_{\alpha }\in \mathcal{D}^{\prime }\left( E\right) $
($E$ is Hilbert),%
\begin{equation*}
u\diamond v=\sum_{\alpha }\sum_{\beta \leq \alpha }(u_{\beta },v_{\alpha
-\beta })_{E}\mathfrak{N}_{\alpha }.
\end{equation*}

For a generalized random field on $u=\sum_{\alpha }u_{\alpha }\mathfrak{N}%
_{\alpha }\in \mathcal{D}^{\prime }\left( \mathbf{H}\left( Y\right) \right)
, $ we define its Skorokhod integrals: 
\begin{equation*}
\delta _{p}(u)=\int_{U}u(\upsilon )\diamond \mathfrak{N}_{p}\left( d\upsilon
\right) =\sum_{\alpha }\int_{U}a_{\alpha }(\upsilon )E_{p}(\upsilon )d\mu 
\mathfrak{N}_{\alpha +p},\left\vert p\right\vert =1,
\end{equation*}%
and 
\begin{eqnarray*}
\delta (u) &=&\int_{U}u\left( \upsilon \right) \diamond \mathfrak{N}\left(
d\upsilon \right) =\int_{U}u\left( \upsilon \right) \diamond \mathfrak{\dot{N%
}}\left( \upsilon \right) \mu \left( d\upsilon \right) =\sum_{\left\vert
p\right\vert =1}\delta _{p}\left( u\right) \\
&=&\sum_{\alpha }\sum_{\left\vert p\right\vert =1}\int_{U}a_{\alpha
}(x)E_{p}(x)d\mu \mathfrak{N}_{\alpha +p}=\sum_{|\alpha |\geq
1}\sum_{\left\vert p\right\vert =1}\int_{U}a_{\alpha -p}(x)E_{p}(x)d\mu 
\mathfrak{N}_{\alpha }.
\end{eqnarray*}%
Note that for a deterministic $u=u\mathfrak{N}_{0}\in \mathbf{H}\left(
Y\right) ,$%
\begin{equation*}
\delta _{\varepsilon _{k}}(u)=\int_{U}u(\upsilon )m_{k}(\upsilon )d\mu 
\mathfrak{N}_{\varepsilon _{k}},\delta \left( u\right)
=\sum_{k}\int_{U}u(\upsilon )m_{k}(\upsilon )d\mu \mathfrak{N}_{\varepsilon
_{k}}=\mathfrak{N}\left( u\right) .
\end{equation*}

Now, we describe Skorokhod integral in terms of the multiple integrals $%
I_{n} $. We show that $\delta :\mathcal{S}\rightarrow \mathcal{S}$ and $%
\delta :\mathcal{S}^{\prime }\rightarrow \mathcal{S}^{\prime }.$

\begin{proposition}
\label{prop1}Let 
\begin{equation*}
u=\sum_{n=0}^{\infty }\frac{1}{n!}I_{n}\left( u_{n}\right) \in \mathcal{S}%
^{\prime }\left( L^{2}\left( U,d\mu \right) \right) ,
\end{equation*}%
i.e., 
\begin{equation*}
u_{n}=u_{n}\left( \upsilon ;\upsilon _{1},\ldots ,\upsilon _{n}\right)
=\sum_{\left\vert \alpha \right\vert =n}u_{\alpha }(\upsilon )E_{\alpha
}\left( \upsilon _{1},\ldots ,\upsilon _{n}\right) ,
\end{equation*}%
with%
\begin{equation*}
\sum_{|\alpha |=n}\int \left\vert u_{\alpha }\left( \upsilon \right)
\right\vert ^{2}d\mu \alpha !<\infty \text{ }\forall n\text{.}
\end{equation*}%
Then 
\begin{equation*}
\delta \left( u\right) =\int u\left( \upsilon \right) \diamond \mathfrak{N}%
\left( d\upsilon \right) =\sum_{n=0}^{\infty }\frac{1}{n!}I_{n+1}(\tilde{F}%
_{n}),
\end{equation*}%
where $\tilde{F}_{n}$ is the standard symmetrization of $F_{n}$ on $U^{n+1}$.
\end{proposition}

\begin{proof}
According to Remark \ref{re2}, for $\left\vert \alpha \right\vert =n,$%
\begin{equation*}
u_{\alpha }\left( \upsilon \right) =\frac{1}{\alpha !n!}\int_{U^{n}}u_{n}(%
\upsilon ;\upsilon ^{\prime })E_{\alpha }\left( \upsilon ^{\prime }\right)
\mu _{n}\left( d\upsilon ^{\prime }\right)
\end{equation*}%
and%
\begin{equation*}
u=\sum_{\alpha }u_{\alpha }(\upsilon )\mathfrak{N}_{\alpha }.
\end{equation*}%
Note that%
\begin{equation*}
\int_{U\times U^{n}}u_{n}\left( \upsilon ;\upsilon ^{\prime }\right)
^{2}d\mu _{n+1}<\infty
\end{equation*}%
and%
\begin{eqnarray*}
\int_{U}u_{\alpha }(\upsilon )E_{p}(\upsilon )d\mu &=&\frac{1}{\alpha !n!}%
\int_{U}\int_{U^{n}}u_{n}(r;\upsilon ^{\prime })E_{p}(r)\mu \left( dr\right)
E_{\alpha }\left( \upsilon ^{\prime }\right) \mu _{n}\left( d\upsilon
^{\prime }\right) \\
&=&\frac{1}{\sqrt{\alpha !n!}}\int_{U^{n}}\int_{U}u_{n}(r;\upsilon ^{\prime
})E_{p}(r)\mu \left( dr\right) e_{\alpha }\left( \upsilon ^{\prime }\right)
\mu _{n}\left( d\upsilon ^{\prime }\right)
\end{eqnarray*}%
and%
\begin{equation*}
u_{n}(\upsilon ,\upsilon ^{\prime })=\sum_{\left\vert \alpha \right\vert
=n,\left\vert p\right\vert =1}\int_{U}u_{\alpha }(r)E_{p}(r)d\mu E_{p}\left(
\upsilon \right) E_{\alpha }(\upsilon ^{\prime }).
\end{equation*}%
Therefore, the standard symmetrization of $u_{n}(\upsilon ;,\upsilon
^{\prime }),\upsilon \in U,\upsilon ^{\prime }=\left( \upsilon _{1},\ldots
,\upsilon _{n}\right) \in U^{n}$, is 
\begin{equation*}
\tilde{u}_{n}(\upsilon ,\upsilon ^{\prime })=\sum_{\left\vert \alpha
\right\vert =n+1,\left\vert p\right\vert =1}\int_{U}u_{\alpha
-p}(r)E_{p}(r)d\mu \frac{n!}{(n+1)!}E_{\alpha }(\upsilon ,\upsilon ^{\prime
}).
\end{equation*}%
By definition of the Skorokhod integral$,$%
\begin{eqnarray*}
\delta \left( u\right) &=&\sum_{\alpha }\sum_{\left\vert p\right\vert
=1}\int_{U}a_{\alpha }(r)E_{p}(r)d\mu \mathfrak{N}_{\alpha +p} \\
&=&\sum_{n=0}^{\infty }\sum_{\left\vert \alpha \right\vert
=n}\sum_{\left\vert p\right\vert =1}\int_{U}a_{\alpha }(r)E_{p}(r)d\mu \frac{%
I_{n+1}(E_{\alpha +p})}{(n+1)!} \\
&=&\sum_{n=0}^{\infty }\sum_{\left\vert \alpha \right\vert
=n+1}\sum_{\left\vert p\right\vert =1}\int_{U}a_{\alpha -p}(r)E_{p}(r)d\mu 
\frac{I_{n+1}(E_{\alpha })}{(n+1)!}=\sum_{n=0}^{\infty }\frac{1}{n!}I_{n+1}(%
\tilde{F}_{n}),
\end{eqnarray*}%
and the statement follows.
\end{proof}

\textbf{Multiple Skorokhod integrals}

For symmetric $u=\sum_{\alpha }u_{\alpha }\mathfrak{N}_{\alpha }\in \mathcal{%
D}^{\prime }\left( \mathbf{H}^{\hat{\otimes}n}\right) ,$ with $u_{\alpha
}\in \mathbf{H}^{\hat{\otimes}n},$ and $\left\vert p\right\vert =n$, we
define%
\begin{eqnarray*}
\delta _{p}\left( u\right) &=&\sum_{\alpha }\int_{U^{n}}u_{\alpha }\frac{%
E_{p}}{p!}d\mu _{n}\mathfrak{N}_{\alpha +p}, \\
\delta ^{n}\left( u\right) &=&\sum_{\left\vert p\right\vert =n}\delta
_{p}\left( u\right) =\sum_{\alpha }\sum_{\left\vert p\right\vert
=n}\int_{U^{n}}u_{\alpha }\frac{E_{p}}{p!}d\mu _{n}\mathfrak{N}_{\alpha +p}.
\end{eqnarray*}%
Let $\delta ^{0}\left( u\right) =u_{0}$. Note that for a deterministic $%
u(x_{1},\ldots ,x_{n})=u(x_{1},\ldots ,x_{n})\mathfrak{N}_{0}$ in $\mathbf{H}%
^{\hat{\otimes}n},$%
\begin{eqnarray*}
\delta ^{n}\left( u\right) &=&\sum_{\left\vert p\right\vert =n}\int_{U^{n}}u%
\frac{E_{p}}{p!}d\mu _{n}\mathfrak{N}_{p}=\sum_{\left\vert p\right\vert
=n}\int_{U^{n}}u\frac{E_{p}}{p!}d\mu _{n}I_{n}\left( E_{p}\right) /n!, \\
\delta ^{n}\left( E_{p}\right) &=&\int_{U^{n}}E_{p}\frac{E_{p}}{p!}d\mu _{n}%
\mathfrak{N}_{p}=n!\mathfrak{N}_{p}=I_{n}(E_{p}),\delta ^{n}\left( u\right)
=I_{n}\left( u\right) .
\end{eqnarray*}

It is easy to show that for $u=\sum_{\alpha }u_{\alpha }\mathfrak{N}_{\alpha
}\in \mathcal{D}^{\prime }\left( \mathbf{H}^{\hat{\otimes}n}\right) ,n\geq
1, $ (with $a_{\alpha }\in \mathbf{H}^{\hat{\otimes}n}$)$,$ $\delta
^{n}\left( u\right) =\delta \left( \delta ^{n-1}\left( \tilde{u}\left(
\upsilon \right) \right) \right) $, where%
\begin{equation*}
\tilde{u}\left( \upsilon \right) =\tilde{u}(\upsilon ;\upsilon _{1},\ldots
,\upsilon _{n-1})=u(\upsilon ,\upsilon _{1},\ldots ,\upsilon
_{n-1}),\upsilon ,\upsilon _{i}\in U\text{.}
\end{equation*}

\begin{remark}
In the framework of a single r.v. $\xi $ (see Remark \ref{r2}),%
\begin{equation*}
\delta ^{k}\left( \mathfrak{N}_{n}\right) =\mathfrak{N}_{n+k},k\geq 1.
\end{equation*}
\end{remark}

\begin{lemma}
\label{le1}For a deterministic $u=u\mathfrak{N}_{0}$ in $\mathbf{H}^{\hat{%
\otimes}n},$%
\begin{equation*}
\mathbf{E}\left[ \delta ^{n}\left( u\right) ^{2}\right] =n!\int_{U^{n}}\left%
\vert u\right\vert ^{2}d\mu _{n}.
\end{equation*}
\end{lemma}

\begin{proof}
Indeed,%
\begin{eqnarray*}
\delta ^{n}\left( u\right) &=&\sum_{\left\vert p\right\vert =n}\int_{U^{n}}u%
\frac{E_{p}}{p!}d\mu _{n}\mathfrak{N}_{p}, \\
\mathbf{E}\left[ \delta ^{n}\left( u\right) ^{2}\right] &=&\sum_{\left\vert
p\right\vert =n}(\int_{U^{n}}u\frac{E_{p}}{p!}d\mu
_{n})^{2}p!=n!\int_{U^{n}}\left\vert u\right\vert ^{2}d\mu _{n}.
\end{eqnarray*}
\end{proof}

\begin{remark}
\label{re1}We can rewrite Proposition \ref{p1} using multiple integrals. For
each $\eta \in L_{2}\left( \Omega ,\mathcal{F}^{0},\mathbf{P}\right) ,$%
\begin{equation*}
\eta =\sum_{\alpha }\eta _{\alpha }\mathfrak{N}_{\alpha }=\sum_{n=0}^{\infty
}\frac{1}{n!}\sum_{|\alpha |=n}\eta _{\alpha }I_{n}\left( E_{\alpha }\right)
=\sum_{n=0}^{\infty }\frac{1}{n!}I_{n}\left( \eta _{n}\right)
=\sum_{n=0}^{\infty }\frac{1}{n!}\delta ^{n}\left( \eta _{n}\right)
\end{equation*}%
with%
\begin{equation*}
\eta _{\alpha }=\frac{\mathbf{E}\left[ \eta \mathfrak{N}_{\alpha }\right] }{%
\alpha !},
\end{equation*}%
and%
\begin{equation*}
\eta _{n}=\eta _{n}\left( \upsilon \right) =\sum_{\left\vert \alpha
\right\vert =n}\eta _{\alpha }E_{\alpha }\left( \upsilon \right) .
\end{equation*}%
Note that%
\begin{equation*}
\eta _{\alpha }=\frac{1}{\alpha !n!}\int_{U^{n}}F_{n}(\upsilon )E_{\alpha
}\left( \upsilon \right) d\mu _{n},\alpha \in J.
\end{equation*}
\end{remark}

\subsection{\textbf{Malliavin derivative}}

We define 
\begin{equation*}
\mathbb{D}\mathfrak{N}_{\alpha }=\sum_{\left\vert p\right\vert =1,p\leq
\alpha }\frac{\alpha !}{(\alpha -p)!}\mathfrak{N}_{\alpha -p}E_{p}\left(
\upsilon \right) =\sum_{\gamma }\sum_{\left\vert p\right\vert =1,\gamma
+p=\alpha }\frac{(\gamma +p)!}{\gamma !}E_{p}\left( \upsilon \right) 
\mathfrak{N}_{\gamma }.
\end{equation*}%
For $u=\sum_{\alpha }u_{\alpha }\mathfrak{N}_{\alpha }\in \mathcal{D}$, 
\begin{eqnarray*}
\mathbb{D}_{k}u &=&\sum_{|\alpha |\geq 1}\alpha _{k}u_{\alpha }\mathfrak{N}%
_{\alpha (k)}m_{k}\left( \upsilon \right) ,\mathbb{D}_{p}u=\sum_{\alpha \geq
p}\frac{\alpha !}{(\alpha -p)!}u_{\alpha }\mathfrak{N}_{\alpha
-p}E_{p}\left( \upsilon \right) ,\left\vert p\right\vert =1, \\
\mathbb{D}u &=&\sum_{\left\vert p\right\vert =1}\mathbb{D}_{p}u=\sum_{\alpha
\geq p}\sum_{\left\vert p\right\vert =1}\frac{\alpha !}{(\alpha -p)!}%
u_{\alpha }\mathfrak{N}_{\alpha -p}E_{p}\left( \upsilon \right) \\
&=&\sum_{\alpha }\sum_{\left\vert p\right\vert =1}\frac{(\alpha +p)!}{\alpha
!}u_{\alpha +p}E_{p}\left( \upsilon \right) \mathfrak{N}_{\alpha }.
\end{eqnarray*}

In a standard way we define the higher order Malliavin derivatives: for $%
u=\sum_{\alpha }a_{\alpha }\mathfrak{N}_{\alpha }\in \mathcal{D}$, 
\begin{eqnarray*}
\mathbb{D}_{p}^{n}u &=&\sum_{\alpha \geq p}\frac{\alpha !}{(\alpha -p)!}%
u_{\alpha }\mathfrak{N}_{\alpha -p}\frac{E_{p}\left( \upsilon _{1},\ldots
,\upsilon _{n}\right) }{p!},\left\vert p\right\vert =n, \\
\mathbb{D}^{n}u &=&\mathbb{D}_{p}^{n}u=\sum_{\left\vert p\right\vert
=n}\sum_{\alpha \geq p}\frac{\alpha !}{(\alpha -p)!p!}u_{\alpha }\mathfrak{N}%
_{\alpha -p}E_{p}(\upsilon _{1},\ldots ,\upsilon _{n}) \\
&=&\sum_{\alpha }\sum_{\left\vert p\right\vert =n}\frac{(\alpha +p)!}{\alpha
!p!}u_{\alpha +p}E_{p}(\upsilon _{1},\ldots ,\upsilon _{n})\mathfrak{N}%
_{\alpha }.
\end{eqnarray*}

We define Malliavin derivative for multiple integrals as well.

\begin{proposition}
\label{prop2}Let $v=\sum_{\left\vert \alpha \right\vert =n}v_{\alpha
}E_{\alpha }\in \mathbf{H}^{\hat{\otimes}n}(Y)$ has only finite number of $%
v_{\alpha }\neq 0.$ Then%
\begin{equation*}
\mathbb{D}I_{n}\left( u\right) =nI_{n-1}\left( u\left( \cdot ,t\right)
\right) ,t\in U.
\end{equation*}
\end{proposition}

\begin{proof}
By definition,%
\begin{equation*}
I_{n}\left( v\right) =\sum_{\left\vert \alpha \right\vert =n}v_{\alpha
}I_{n}\left( E_{\alpha }\right) =\sum_{\left\vert \alpha \right\vert
=n}v_{\alpha }n!\mathfrak{N}_{\alpha }\in \mathcal{D}.
\end{equation*}

Since%
\begin{equation*}
v_{\alpha }=\frac{1}{\alpha !n!}\int_{U^{n}}vE_{\alpha }d\mu _{n},
\end{equation*}%
and for $t,\upsilon _{1},\ldots ,\upsilon _{n}\in U,$%
\begin{equation*}
v(t,\upsilon _{1},\ldots ,\upsilon _{n-1})=\sum_{\left\vert p\right\vert
=1}\int_{U}v(t^{\prime },\upsilon _{1},\ldots ,\upsilon _{n-1})E_{p}\left(
t^{\prime }\right) d\mu E_{p}\left( t\right)
\end{equation*}%
(it is a finite sum), its Malliavin derivative ($\left\{ \left( \alpha
,p\right) :\left\vert \alpha +p\right\vert =n,\left\vert p\right\vert
=1\right\} =\left\{ \left( \alpha ,p\right) :\left\vert \alpha \right\vert
=n-1,\left\vert p\right\vert =1\right\} $) is%
\begin{eqnarray*}
&&\mathbb{D}I_{n}\left( v\right) \\
&=&\sum_{\alpha }\sum_{\left\vert p\right\vert =1,\left\vert \alpha
+p\right\vert =n}\frac{\left( \alpha +p\right) !}{\alpha !}v_{\alpha
+p}n!E_{p}\left( \upsilon \right) \mathfrak{N}_{\alpha }=\sum_{\left\vert
\alpha \right\vert =n-1}\sum_{\left\vert p\right\vert =1}\frac{\left( \alpha
+p\right) !}{\alpha !}v_{\alpha +p}n!E_{p}\left( \upsilon \right) \mathfrak{N%
}_{\alpha } \\
&=&\sum_{\left\vert \alpha \right\vert =n-1}\sum_{\left\vert p\right\vert =1}%
\frac{1}{\alpha !}\int vE_{\alpha +p}d\mu _{n}E_{p}\left( \upsilon \right) 
\mathfrak{N}_{\alpha }=\sum_{\left\vert \alpha \right\vert
=n-1}\sum_{\left\vert p\right\vert =1}n\int v\frac{E_{\alpha }}{\alpha !}%
E_{p}d\mu _{n}E_{p}\left( t\right) \mathfrak{N}_{\alpha } \\
&=&\sum_{\left\vert \alpha \right\vert =n-1}n(n-1)!\int v(t,\cdot )\frac{%
E_{\alpha }}{\alpha !(n-1)!}d\mu _{n-1}\mathfrak{N}_{\alpha
}=\sum_{\left\vert \alpha \right\vert =n-1}n(n-1)!v_{\alpha }(t,\cdot )%
\mathfrak{N}_{\alpha } \\
&=&nI_{n-1}\left( v\left( \cdot ,t\right) \right) .
\end{eqnarray*}%
We used here that%
\begin{equation*}
\widetilde{E_{\alpha }E_{p}}=\frac{\left\vert \alpha \right\vert !}{%
\left\vert \alpha +p^{\prime }\right\vert !}E_{\alpha +p}=\frac{1}{n}%
E_{\alpha +p},
\end{equation*}%
where $\widetilde{f}$ is the symmetrization of $f$.
\end{proof}

As suggested by Proposition \ref{prop2}, for an arbitrary $%
v=\sum_{\left\vert \alpha \right\vert =n}v_{\alpha }E_{\alpha }\in \mathbf{H}%
^{\hat{\otimes}n}(Y)$ we define%
\begin{equation*}
\mathbb{D}I_{n}\left( v\right) =nI_{n-1}\left( v\left( y,\cdot \right)
\right) ,y\in U,
\end{equation*}%
$I_{0}\left( c\right) =c,c\in \mathbf{R}$. For $u=\sum_{n}\frac{I_{n}(u_{n})%
}{n!}\in \mathcal{S}^{\prime }\left( Y\right) $ we define%
\begin{equation*}
\mathbb{D}u\left( y\right) =\sum_{n}\frac{\mathbb{D}I_{n}(u_{n})}{n!}%
=\sum_{n}\frac{I_{n-1}(u_{n}(y,\cdot ))}{(n-1)!}
\end{equation*}

We see that $\mathbb{D}$ maps $\mathcal{S}\left( Y\right) $ into $\mathcal{S}%
\left( Y\right) $.

\begin{remark}
\label{r3}In the framework of a single r.v. $\xi $ (see Remark \ref{r2}),$%
\mathbb{D}^{k}\left( \xi ^{\diamond n}\right) =\frac{n!}{(n-k)!}\xi
^{\diamond (n-k)},k\geq 1.$
\end{remark}

\subsection{Adapted stochastic processes}

In this subsection we assume that $U=[0,T]\times V,\mathcal{U}=\mathcal{B}%
\left( [0,T]\right) \times \mathcal{V},d\mu =dtd\pi .$ Let 
\begin{equation}
u=\sum_{n}\frac{I_{n}\left( u_{n}\right) }{n!}\in \mathcal{S}^{\prime
}\left( Y\right) ,  \label{01}
\end{equation}%
with $u_{n}\in \mathbf{H}^{\hat{\otimes}n}(Y),n\geq 0$ ($u_{n}=u_{n}\left(
t_{1},\upsilon _{1},\ldots ,t_{n},\upsilon _{n}\right) ,\left(
t_{i},\upsilon _{i}\right) \in U,i=1,\ldots ,n$)$.$

For $t\in \lbrack 0,T]$, let $Q_{t}^{n}=\left( [0,t]\times V\right) ^{n}$.

\begin{definition}
Let $t_{0}\in \left[ 0,T\right] $. A random variable $u$ defined by $\left( 
\text{\ref{01}}\right) $ is called $\mathcal{F}_{t_{0}}$-measurable if for
each $n,$ supp$\left( u_{n}\right) \subseteq Q_{t_{0}}^{n}\,,\,$i.e. $\mu
_{n-1}$-a.e.%
\begin{equation*}
u_{n}\left( t_{1},\upsilon _{1},\ldots ,t_{n},\upsilon _{n}\right) =0\text{
if }t_{i}>t_{0}\text{ for some }i
\end{equation*}
\end{definition}

\begin{proposition}
\label{prop3}A random variable $u\in \mathcal{S}\left( Y\right) $ defined by 
$\left( \text{\ref{01}}\right) $ is $\mathcal{F}_{t_{0}}$- measurable iff%
\begin{equation*}
\mathbb{D}u\left( t,\upsilon \right) =0\text{ if }t>t_{0}\text{.}
\end{equation*}
\end{proposition}

\begin{proof}
For $u\in \mathcal{S}\left( Y\right) $ defined by (\ref{01}), 
\begin{equation}
\mathbb{D}u=\sum_{n\geq 1}\frac{1}{(n-1)!}I_{n-1}\left( u_{n}(t,\upsilon
,\cdot \right) ,(t,\upsilon )\in U.  \label{02}
\end{equation}%
and%
\begin{eqnarray*}
\mathbf{E}\left[ \mathbb{D}u\left( t,\upsilon \right) ^{2}\right] &=&\sum_{n}%
\frac{1}{(n-1)!}\mathbf{E}[I_{n-1}\left( u_{n}\left( t,\upsilon ,\cdot
\right) \right) ^{2}] \\
&=&\sum_{n}\int_{U^{n-1}}u_{n}\left( t,\upsilon ,\cdot \right) ^{2}d\mu
_{n-1}=0
\end{eqnarray*}%
for $t>t_{0}$ iff $u_{n}(t,\upsilon ,\cdot )=0$ $\mu _{n-1}$-a.e for $%
t>t_{0} $.
\end{proof}

Now, we will introduce a notion of an adapted random process. Consider $u\in
S\left( U;Y\right) $, i.e,%
\begin{equation}
u\left( t,\upsilon \right) =\sum_{n}\frac{I_{n}\left( u_{n}\left( t,\upsilon
,\cdot \right) \right) }{n!}  \label{03}
\end{equation}%
with $u_{n}\left( t,\upsilon ,\cdot \right) \in \mathbf{H}^{\hat{\otimes}%
n}(Y)$ for all $(t,\upsilon )\in U$:%
\begin{equation}
u_{n}(t,\upsilon ,\cdot )=\sum_{\left\vert \alpha \right\vert =n}u_{\alpha
}(t,\upsilon )E_{\alpha }\left( \cdot \right) ,\left( t,\upsilon \right) \in
U.  \label{04}
\end{equation}

\begin{definition}
A random field $u(t,\upsilon )$ on $U$ defined by (\ref{03}) is called
adapted if supp$\left( u_{n}(t,\upsilon ,\cdot )\right) \subseteq
Q_{t}^{n},\upsilon \in V,$ for every $t\in \lbrack 0,T].$
\end{definition}

A straightforward consequence of Proposition \ref{prop3} is the following
claim.

\begin{corollary}
\label{co1}A random field $u\in S\left( U;Y\right) $ is adapted iff for each 
$t\in \lbrack 0,T]$, the Malliavin derivative $\mathbb{D}u(t,\upsilon
;s_{1},\upsilon _{1},\ldots ,s_{n},\upsilon _{n})=0,\upsilon \in V,$ if $%
s_{i}>t$ for some $i.$
\end{corollary}

Given a random field $u(t,\upsilon )$ on $U=[0,T]\times V$, consider its
Skorokhod integral%
\begin{equation*}
\delta \left( u\right) _{t}=\delta \left( \chi _{\lbrack 0,t]}u\right)
=\int_{0}^{t}u(r,y)\diamond \mathfrak{N}(dr,dy),0\leq t\leq T.
\end{equation*}

\begin{proposition}
\label{prop4}Consider a random field $u\in S\left( \mathbf{H}\left( Y\right)
;Y\right) $, i.e. (\ref{03}) holds with 
\begin{equation*}
\int_{U^{n+1}}\left\vert u_{n}\right\vert _{Y}^{2}d\mu _{n+1}<\infty \text{ }%
\forall n.
\end{equation*}%
If it is adapted, then $\delta \left( u\right) _{t},0\leq t\leq T,$ is
adapted as well.
\end{proposition}

\begin{proof}
Since $u\left( t,\upsilon \right) =\sum_{n}\frac{I_{n}\left( u_{n}\left(
t,\upsilon ,\cdot \right) \right) }{n!}$ is adapted with $u_{n}$ satisfying (%
\ref{04}), supp$\left( u_{n}(t,\upsilon ,\cdot \right) )\subseteq
Q_{t}^{n},\upsilon \in V,$ for all $n$ and $t\in \lbrack 0,T]$, i.e $%
u_{n}\left( t,\upsilon \right) =u_{n}(t,\upsilon )\chi _{Q_{t}^{n}}$. By
Proposition \ref{prop1},%
\begin{equation*}
\delta \left( u\right) _{t}=\delta \left( \chi _{\lbrack 0,t]}u\right)
=\sum_{n=0}^{\infty }\frac{1}{n!}I_{n+1}(\widetilde{\chi _{\lbrack 0,t]}u_{n}%
}),
\end{equation*}%
where $\widetilde{\chi _{\lbrack 0,t]}u_{n}}$ is the standard symmetrization
of $\chi _{\lbrack 0,t]}u_{n}=u_{n}\chi _{Q_{t}^{n+1}}.$ Since its support
is obviously a subset of $Q_{t}^{n+1}$, the statement follows.
\end{proof}

\subsection{Ito-Skorokhod isometry}

Now we estimate the $L_{2}$-norm of the Skorokhod integral.

\begin{proposition}
\label{p2}Let $u=u\left( \upsilon \right) =\sum_{\alpha }u_{\alpha }\left(
\upsilon \right) \mathfrak{N}_{\alpha }\in L_{2}(\mathbb{D}\left( U;Y\right)
,d\mu ),$ i.e. $u_{\alpha }\in \mathbf{H}\left( Y\right) :$%
\begin{equation*}
\int |u_{a}\left( \upsilon \right) |_{Y}^{2}d\mu <\infty ,\alpha \in J,
\end{equation*}%
with a finite number of $u_{\alpha }\neq 0$. Then%
\begin{equation*}
\mathbf{E}\left[ |\delta (u)|_{Y}^{2}\right] =\mathbf{E}\left[
\int_{U}\left\vert u(\upsilon )\right\vert _{Y}^{2}d\mu \right] +\mathbf{E}%
\left[ \int_{U^{2}}(\mathbb{D}u(\upsilon ;\upsilon ^{\prime }),\mathbb{D}%
u(\upsilon ^{\prime };\upsilon ))_{Y}\mu (d\upsilon )\mu (d\upsilon ^{\prime
})\right] ,
\end{equation*}%
where%
\begin{equation*}
\mathbb{D}u(\upsilon ;\upsilon ^{\prime })=\sum_{\alpha }\sum_{\left\vert
p\right\vert =1}\frac{(\alpha +p)!}{\alpha !}u_{\alpha +p}(\upsilon
)E_{p}\left( \upsilon ^{\prime }\right) \mathfrak{N}_{\alpha }.
\end{equation*}
\end{proposition}

\begin{proof}
By definition,%
\begin{equation*}
\delta (u)=\sum_{|\alpha |\geq 1}\sum_{\left\vert p\right\vert
=1}\int_{U}u_{\alpha -p}(x)E_{p}(x)d\mu \mathfrak{N}_{\alpha }=\sum_{\alpha
}\sum_{\left\vert p\right\vert =1}\int_{U}u_{\alpha }(r)E_{p}(r)d\mu 
\mathfrak{N}_{\alpha +p}.
\end{equation*}%
Hence%
\begin{eqnarray*}
\mathbf{E}\left[ |\delta (u)|_{Y}^{2}\right] &=&\sum_{|\alpha |\geq
1}\left\vert \sum_{\left\vert p\right\vert =1}\int_{U}u_{\alpha
-p}(x)E_{p}(x)d\mu \right\vert _{Y}^{2}\alpha ! \\
&=&\sum_{\left\vert p\right\vert =\left\vert p^{\prime }\right\vert
=1}\sum_{|\alpha |\geq 1}\int_{U^{2}}(u_{\alpha -p}(x),u_{\alpha -p^{\prime
}}(x^{\prime }))_{Y}E_{p}(x)E_{p^{\prime }}(x^{\prime })\mu (dx)\mu \left(
dx^{\prime }\right) \alpha ! \\
&=&\sum_{p=p^{\prime },\left\vert p\right\vert =1}...+\sum_{p\neq p^{\prime
},\left\vert p\right\vert =\left\vert p^{\prime }\right\vert =1}...=A+B.
\end{eqnarray*}%
Now%
\begin{eqnarray*}
A &=&\sum_{\left\vert p\right\vert =1}\sum_{\alpha \geq p}|\int_{U}u_{\alpha
-p}(x)E_{p}(x)d\mu |_{Y}^{2}\alpha !=\sum_{\gamma }\sum_{\left\vert
p\right\vert =1}|\int_{U}u_{\gamma }(x)E_{p}(x)d\mu |_{Y}^{2}(\gamma +p)! \\
&=&\sum_{\gamma }\sum_{\left\vert p\right\vert =1}|\int_{U}u_{\gamma
}(x)E_{p}(x)d\mu |_{Y}^{2}\gamma !+\sum_{\gamma \geq p}\sum_{\left\vert
p\right\vert =1}|\int_{U}u_{\gamma }(x)E_{p}(x)d\mu |_{Y}^{2}[(\gamma
+p)!-\gamma !] \\
&=&\sum_{\gamma }\int_{U}\gamma !|u_{\gamma }(x)|_{Y}^{2}d\mu +\sum_{\gamma
\geq p}\sum_{\left\vert p\right\vert =1}|\int_{U}u_{\gamma }(x)E_{p}(x)d\mu
|_{Y}^{2}[(\gamma +p)!-\gamma !].
\end{eqnarray*}%
Also,%
\begin{eqnarray*}
&&B \\
&=&\sum_{p\neq p^{\prime },\left\vert p\right\vert =\left\vert p^{\prime
}\right\vert =1}\sum_{\alpha \geq p+p^{\prime }}\int_{U^{2}}(u_{\alpha
-p}(x),u_{\alpha -p^{\prime }}(x^{\prime }))_{Y}E_{p}(x)E_{p^{\prime
}}(x^{\prime })\mu (dx)\mu \left( dx^{\prime }\right) \alpha ! \\
&=&\sum_{p\neq p^{\prime },\left\vert p\right\vert =\left\vert p^{\prime
}\right\vert =1}\sum_{\beta \geq 0}\int_{U^{2}}(u_{\beta +p^{\prime
}}(x),u_{\beta +p}(x^{\prime }))_{Y}\times \\
&&\times E_{p}(x)E_{p^{\prime }}(x^{\prime })\mu (dx)\mu \left( dx^{\prime
}\right) (\beta +p+p^{\prime })!.
\end{eqnarray*}%
On the other hand,%
\begin{eqnarray*}
&&\mathbf{E}\left[ \mathbb{D}u(x;x^{\prime }),\mathbb{D}_{x}u(x^{\prime
};x))_{Y}\right] \\
&=&\sum_{\alpha }\sum_{\left\vert p\right\vert =|p^{\prime }|=1}\frac{%
(\alpha +p)!}{\alpha !}(u_{\alpha +p}(x),a_{\alpha +p^{\prime }}(x^{\prime
}))_{Y}\frac{(\alpha +p^{\prime })!}{\alpha !}E_{p}\left( x^{\prime }\right)
E_{p^{\prime }}\left( x\right) \alpha ! \\
&=&\sum_{\alpha }\sum_{p=p^{\prime },\left\vert p\right\vert
=1}...+\sum_{\alpha }\sum_{p\neq p^{\prime },\left\vert p\right\vert
=\left\vert p^{\prime }\right\vert =1}...=C+D.
\end{eqnarray*}%
Obviously,%
\begin{equation*}
C=\sum_{|p|=1}\sum_{\alpha \geq p}\frac{\alpha !}{(\alpha -p)!}(u_{\alpha
}(x),u_{\alpha }(x^{\prime }))_{Y}E_{p}\left( x\right) E_{p}(x^{\prime
})\alpha !.
\end{equation*}

Comparing%
\begin{equation*}
D=\sum_{\alpha }\sum_{p\neq p^{\prime },\left\vert p\right\vert =|p^{\prime
}|=1}\frac{(\alpha +p)!}{\alpha !}(a_{\alpha +p}(x)E_{p}\left( x^{\prime
}\right) ,a_{\alpha +p^{\prime }}(x^{\prime })E_{p^{\prime }}\left( x\right)
)_{Y}\frac{(\alpha +p^{\prime })!}{\alpha !}\alpha !,
\end{equation*}%
$\int_{U^{2}}Cd\mu _{2},\int_{U^{2}}Dd\mu _{2}$ and $A,B$, the statement
follows.
\end{proof}

\begin{corollary}
\label{c1}Let $u\in \mathcal{S}\left( \mathbf{H}\left( Y\right) ,Y\right) $,
i.e. (\ref{03}) holds with 
\begin{equation*}
\int_{U^{n+1}}\left\vert u_{n}\right\vert _{Y}^{2}d\mu _{n+1}<\infty \text{ }%
\forall n.
\end{equation*}

Then the statement of Proposition \ref{p2} holds for $\delta \left( u\right)
.$
\end{corollary}

\begin{proof}
It is enough to prove the statement for $u\left( \upsilon \right)
=I_{n}\left( u_{n}\left( \upsilon \right) \right) $, where 
\begin{equation*}
u_{n}\left( \upsilon ,\cdot \right) =\sum_{\left\vert \alpha \right\vert
=n}u_{\alpha }\left( \upsilon \right) E_{\alpha }\left( \cdot \right)
\end{equation*}%
with with a finite number nonzero $u_{\alpha }\in \mathbf{H}\left( Y\right)
: $ 
\begin{equation*}
\int_{U}\left\vert u_{\alpha }\right\vert _{Y}^{2}d\mu <\infty .
\end{equation*}%
In this case, $u=n!\sum_{\left\vert \alpha \right\vert =n}u_{\alpha }\left(
\upsilon \right) \mathfrak{N}_{\alpha }\in L_{2}(\mathbb{D}\left( U;Y\right)
,d\mu )$ and Proposition \ref{p2} applies. We obtain the general case by
linearity and passing to the limit.
\end{proof}

\begin{remark}
In the framework of a single r.v. $\xi $ (see Remark \ref{r2}), for $%
u=\sum_{n}u_{n}\xi ^{\diamond n}$ we have%
\begin{equation*}
\mathbf{E}\left[ \delta (u)^{2}\right] =\mathbf{E}\left[ \left\vert
u\right\vert ^{2}\right] +\mathbf{E}\left[ (\mathbb{D}u)^{2}\right] .
\end{equation*}
\end{remark}

For an adapted random field on $U=[0,T]\times V,d\mu =dtd\pi $, the standard
isometry holds. It is an obvious consequence of Corollary \ref{c1}

\begin{corollary}
\label{c2}Let $\mathbf{H}=L^{2}\left( [0,T]\times V,dtd\pi \right) $. Assume 
$u\in \mathcal{S}\left( \mathbf{H}\left( Y\right) ,Y\right) $ is an adapted
random field on $U=\left[ 0,T\right] \times V$. Then%
\begin{equation*}
\mathbf{E}\left[ |\delta (u)|_{Y}^{2}\right] =\mathbf{E}\left[
\int_{U}\left\vert u(t,\upsilon )\right\vert _{Y}^{2}d\mu \right] .
\end{equation*}
\end{corollary}

\textbf{Duality between }$\delta $ and $\mathbb{D}$

\begin{proposition}
\label{p3}Let $u=u\left( x\right) =\sum_{\alpha }u_{\alpha }\left( x\right) 
\mathfrak{N}_{\alpha }\in L_{2}(\mathcal{D}\left( U;\mathbf{R}\right) ,d\mu
),$ i.e. 
\begin{equation*}
\int |u_{a}\left( x\right) |^{2}d\mu <\infty ,\alpha \in J,
\end{equation*}%
with a finite number of $u_{\alpha }\neq 0$. Let $v=\sum_{\alpha }v_{\alpha }%
\mathfrak{N}_{\alpha }\in \mathcal{D}$. Then%
\begin{equation*}
\mathbf{E[}\delta \left( u\right) v]=\mathbf{E}\left[ \int_{U}u(x)\mathbb{D}%
v(x)d\mu \right] .
\end{equation*}
\end{proposition}

\begin{proof}
Indeed, 
\begin{eqnarray*}
&&\mathbf{E[}\delta \left( u\right) v] \\
&=&\sum_{|\alpha |\geq 1}\sum_{\left\vert p\right\vert =1}\int_{U}u_{\alpha
-p}(x)E_{p}(x)d\mu v_{\alpha }\alpha ! \\
&=&\sum_{\left\vert p\right\vert =1}\sum_{\alpha \geq p}\int_{U}u_{\alpha
-p}(x)E_{p}(x)d\mu v_{\alpha }\alpha ! \\
&=&\sum_{\left\vert p\right\vert =1}\sum_{\gamma }\int_{U}u_{\gamma
}(x)E_{p}(x)d\mu v_{\gamma +p}\frac{\left( \gamma +p\right) !}{\gamma !}%
\gamma !=\mathbf{E}\left[ \int_{U}u(x)\mathbb{D}v(x)d\mu \right] .
\end{eqnarray*}
\end{proof}

\section{\protect\bigskip Stochastic Differential Equations}

\subsubsection{Wick exponent}

We start with the definition of Wick exponent.

Let $f=\sum_{k}f_{k}m_{k}\in L^{2}\left( U,d\mu \right) $. For $n\geq 1$ and 
$\mathfrak{N}(f)=\sum_{k}f_{k}\xi _{k}$ we have (denoting $f^{\alpha
}=\prod_{k}f_{k}^{\alpha _{k}}$) 
\begin{equation*}
\mathfrak{N}\left( f\right) ^{\diamond n}=\mathfrak{N}(f)\diamond \ldots
\diamond \mathfrak{N}(f)\text{ (n times)}=\sum_{\left\vert \alpha
\right\vert =n}\frac{n!}{\alpha !}f^{\alpha }\mathfrak{N}_{\alpha }.
\end{equation*}%
Note that $\mathfrak{N}\left( f\right) ^{\diamond n}\in L^{2}\left( \Omega
\right) :$%
\begin{eqnarray*}
\mathbf{E}\left[ \left( \frac{1}{n!}\mathfrak{N}\left( f\right) ^{\diamond
n}\right) ^{2}\right] &=&\sum_{\left\vert \alpha \right\vert =n}\frac{%
z^{2\alpha }}{\alpha !}=\frac{1}{n!}\sum_{\left\vert \alpha \right\vert =n}%
\frac{n!z^{2\alpha }}{\alpha !}=\frac{1}{n!}\left( \sum_{i}z_{i}^{2}\right)
^{n} \\
&=&\frac{1}{n!}\left\vert f\right\vert _{L^{2}\left( \mu \right)
}^{2n}<\infty .
\end{eqnarray*}

Let $\mathcal{Z}$ be the set of all number sequences $z=\left( z_{k}\right) $
with finite number of nonzero terms. The following statement holds.

\begin{proposition}
\label{prop6}a) Let $f=\sum_{k}f_{k}m_{k}\in L^{2}\left( \mu \right) $. Then%
\begin{equation*}
\mathfrak{N}\left( f\right) ^{\diamond n}=I_{n}\left( f^{\otimes n}\right)
\end{equation*}%
and 
\begin{equation*}
\exp ^{\diamond }\left\{ \mathfrak{N(}f)\right\} :=\sum_{n=0}^{\infty }\frac{%
\mathfrak{N}\left( f\right) ^{\diamond n}}{n!}=\sum_{n=0}^{\infty
}\sum_{\left\vert \alpha \right\vert =n}\frac{f^{\alpha }}{\alpha !}%
\mathfrak{N}_{\alpha }=\sum_{\alpha }\frac{f^{\alpha }}{\alpha !}\mathfrak{N}%
_{\alpha }\in L^{2}\left( \Omega \right)
\end{equation*}%
with%
\begin{equation*}
\frac{f^{\alpha }}{\alpha !}=\frac{1}{n!\alpha !}\int f^{\otimes n}E_{\alpha
}d\mu _{n}.
\end{equation*}%
Moreover,%
\begin{equation}
\mathbf{E}\left[ \left( \exp ^{\diamond }\left\{ \mathfrak{N}f\right\}
\right) ^{2}\right] =\exp \left\{ \left\vert f\right\vert _{L^{2}\left( \mu
\right) }^{2}\right\} .  \label{f1}
\end{equation}

b) Let $z=(z_{k})\in \mathcal{Z}$. Then $\mathbf{P}$-a.s.%
\begin{equation*}
p\left( z\right) =\exp ^{\diamond }\left\{ \mathfrak{N}\left(
\sum_{k}z_{k}m_{k}\right) \right\} =\sum_{\alpha }\frac{z^{\alpha }}{\alpha !%
}\mathfrak{N}_{\alpha },z=\left( z_{k}\right) \in \mathcal{Z},
\end{equation*}%
is analytic in $z$ and 
\begin{equation*}
\frac{\partial ^{\left\vert \alpha \right\vert }p\left( z\right) }{\partial
z^{\alpha }}|_{z=0}=\mathfrak{N}_{\alpha }.
\end{equation*}
\end{proposition}

\begin{proof}
a) In terms of multiple integrals we have 
\begin{eqnarray*}
\mathfrak{N}\left( f\right) ^{\diamond n} &=&\sum_{\left\vert \alpha
\right\vert =n}\frac{n!f^{\alpha }}{\alpha !}\mathfrak{N}_{\alpha
}=\sum_{\left\vert \alpha \right\vert =n}\frac{f^{\alpha }}{\alpha !}%
I_{n}\left( E_{\alpha }\right) = \\
&&\sum_{\left\vert p_{i}\right\vert =1,p_{1}+\ldots +p_{n}=\alpha }\frac{%
n!f^{\alpha }}{\alpha !}I_{n}\left( \widetilde{E_{p_{1}}\ldots E_{p_{n}}}%
\right) \\
&=&I_{n}\left( f^{\otimes n}\right) =I_{n}\left( \sum_{\left\vert \alpha
\right\vert =n}\frac{f^{\alpha }}{\alpha !}E_{\alpha }\right) ,
\end{eqnarray*}%
with 
\begin{equation*}
f^{\otimes n}=\sum_{\left\vert \alpha \right\vert =n}\frac{f^{\alpha }}{%
\alpha !}E_{\alpha },f^{\alpha }=\frac{1}{n!}\int f^{\otimes n}E_{\alpha
}d\mu _{n}.
\end{equation*}

In addition,%
\begin{eqnarray*}
\mathbf{E}\left[ \left( \frac{\mathfrak{N}\left( f\right) ^{\diamond n}}{n!}%
\right) ^{2}\right] &=&\sum_{\left\vert \alpha \right\vert =n}\frac{%
f^{2\alpha }}{\alpha !}=\frac{1}{n!}\sum_{\left\vert \alpha \right\vert =n}%
\frac{n!f^{2\alpha }}{\alpha !}=\frac{1}{n!}\left( \sum_{i}f_{i}^{2}\right)
^{n} \\
&=&\frac{1}{n!}\left\vert f\right\vert _{L^{2}\left( \mu \right) }^{2n}
\end{eqnarray*}%
(note also using multiple integrals: $\mathbf{E}\left( \frac{1}{\left(
n!\right) ^{2}}I_{n}\left( f^{\otimes n}\right) ^{2}\right) =\frac{1}{n!}%
\left\vert f^{\otimes n}\right\vert _{L^{2}(\mu _{n})}^{2}=\frac{1}{n!}%
\left\vert f\right\vert _{L^{2}(\mu )}^{2n})$. Moreover,%
\begin{equation*}
\mathbf{E}\left( \sum_{n}\left( \frac{\mathfrak{N}\left( f\right) ^{\diamond
n}}{n!}\right) ^{2}\right) =\sum_{n}\frac{1}{n!}\left\vert f\right\vert
_{L^{2}\left( \mu \right) }^{2n}=\exp \left\{ \left\vert f\right\vert
_{L^{2}\left( \mu \right) }^{2}\right\} .
\end{equation*}

b) Let $z=\left( z_{k}\right) \in \mathcal{Z}$. Then 
\begin{equation*}
\mathbf{E}\left\vert p(z)\right\vert \leq \sum_{\alpha }\frac{\left\vert
z^{\alpha }\right\vert }{\sqrt{\alpha !}}\mathbf{E}\left\vert \frac{%
\mathfrak{N}_{\alpha }}{\sqrt{\alpha !}}\right\vert \leq \sum_{\alpha }\frac{%
\left\vert z^{\alpha }\right\vert }{\sqrt{\alpha !}}\leq \prod_{k}\sum_{n}%
\frac{|z_{k}|^{n}}{\sqrt{n!}}
\end{equation*}%
and the statement follows.
\end{proof}

In a time dependent case the following statement holds.

\begin{corollary}
\label{lr1}Let $U=\left[ 0,T\right] \times V,~d\mu =dtd\pi $. Let $G\in
L^{2}\left( [0,T]\times V,d\mu \right) $. Consider $M_{t}=\exp ^{\diamond
}\left\{ \mathfrak{N}\left( \chi _{\lbrack s,t]}G\right) \right\} ,0\leq
s\leq t\leq T.$

Then%
\begin{equation*}
M_{t}=\sum_{\alpha }\frac{H(s,t)^{\alpha }}{\alpha !}\mathfrak{N}_{\alpha
}=\sum_{n=0}^{\infty }\frac{I_{n}(H_{n}(s,t))}{n!},s\leq t\leq T,
\end{equation*}%
with%
\begin{eqnarray*}
H_{k}(s,t) &=&\int_{s}^{t}\int G\left( r,\upsilon \right) m_{k}(r,\upsilon
)drd\pi ,H(s,t)^{\alpha }=\prod_{k}H_{k}(s,t)^{\alpha _{k}}, \\
&&\frac{H(s,t)^{\alpha }}{\alpha !}=\frac{1}{n!\alpha !}\int \left( \chi
_{\lbrack s,t]}G\right) ^{\otimes n}E_{\alpha }d\mu _{n}\text{ if }%
\left\vert \alpha \right\vert =n, \\
H_{n}\left( s,t\right) &=&\left( \chi _{\lbrack s,,t]}G\right) ^{\otimes n}.
\end{eqnarray*}%
Moreover, $M$ is adapted (for each $t$, the supp$\left( H_{n}(t\right)
)\subseteq Q_{t}^{n}=(\left[ 0,t\right] \times V)^{n}).$
\end{corollary}

\subsubsection{Linear SDE}

Let $U=\left[ 0,T\right] \times V,d\mu =dtd\pi $. Let $w=\sum_{\alpha
}w_{\alpha }\mathfrak{N}_{\alpha }\in \mathcal{D}^{\prime },~f=\sum_{\alpha
}f_{\alpha }\left( t,\upsilon \right) \mathfrak{N}_{\alpha }\in L_{2}\left( 
\mathcal{D}^{\prime }\left( U\right) ,d\mu \right) $. For $G\in L^{2}\left(
\mu \right) ,$, consider the a non-homogeneous equation%
\begin{equation}
\dot{u}\left( t\right) =\int [u\left( t\right) G\left( t,\upsilon \right)
+f\left( t,\upsilon \right) ]\diamond \mathfrak{\dot{N}}(t,\upsilon )\pi
\left( d\upsilon \right) ,u\left( 0\right) =w,  \label{fe11}
\end{equation}%
that is, equivalently,%
\begin{equation}
u(t)=w+\int_{0}^{t}[u(s)G(s,\upsilon )+f(s,\upsilon )]\diamond \mathfrak{N}%
\left( ds,d\upsilon \right) ,0\leq t\leq T.  \label{12}
\end{equation}%
We seek a solution to (\ref{12}) in the form%
\begin{equation}
u\left( t\right) =\sum_{\alpha }u_{\alpha }(t)\mathfrak{N}_{\alpha },0\leq
t\leq T.  \label{13}
\end{equation}

\begin{lemma}
\label{lel1}Let $w=\sum_{\alpha }w_{\alpha }\mathfrak{N}_{\alpha }\in 
\mathcal{D}^{\prime },~f=\sum_{\alpha }f_{\alpha }\left( t,\upsilon \right) 
\mathfrak{N}_{\alpha }\in L_{2}\left( \mathcal{D}^{\prime }\left( U\right)
,d\mu \right) $. Then there is a unique solution to (\ref{12}) in $C\mathcal{%
D}^{\prime }\left( [0,T];\mathbf{R}\right) $ (Recall $C\mathcal{D}^{\prime
}\left( [0,T];\mathbf{R}\right) $ is the class of all generalized processes $%
u=\sum_{\alpha }u_{\alpha }(t)\mathfrak{N}_{\alpha }$ on $\left[ 0,T\right] $
such that $u_{\alpha }$ is continuous on $\left[ 0,T\right] $ $\forall
\alpha \in J.$). The solution $u$ given by (\ref{13}) has the following
coefficients: $u_{0}\left( t\right) =w_{0},$%
\begin{equation}
u_{\alpha }(t)=w_{\alpha }+\sum_{\left\vert p\right\vert =1}\int_{0}^{t}\int
[u_{\alpha -p}\left( r\right) G(r,\upsilon )+f_{\alpha -p}\left( r,\upsilon
\right) ]E_{p}(r,\upsilon )d\pi dr,0\leq t\leq T.  \label{fe3}
\end{equation}
\end{lemma}

\begin{proof}
We seek the solution $u$ to (\ref{12}) in the form of (\ref{13}) with
continuous coefficients $u_{\alpha }$. Plugging the series (\ref{13}) into (%
\ref{12}) we immediately get the system (\ref{fe3}). Since the system is
triangular, starting with $u_{0}\left( t\right) =w_{0}$ we find unique
continuous $u_{\alpha }\left( t\right) $ for $\left\vert \alpha \right\vert
\geq 1.$
\end{proof}

Let 
\begin{equation*}
H_{k}(t)=\int_{0}^{t}\int G\left( s,\upsilon \right) m_{k}(s,\upsilon
)dsd\pi ,H(t)^{\alpha }=\prod_{k}H_{k}(t)^{\alpha _{k}},\alpha \in J.
\end{equation*}%
For $w=\sum_{\alpha }w_{\alpha }\mathfrak{N}_{\alpha }\in \mathcal{D}%
^{\prime }$, let%
\begin{equation*}
\left\vert \left\vert w\right\vert \right\vert ^{2}=\sum_{\alpha }\left\vert
w_{\alpha }\right\vert ^{2}\alpha !+\sup_{t}\sum_{\alpha }\alpha !\left(
\sum_{\beta \leq \alpha }w_{\alpha -\beta }\frac{H(t)^{\beta }}{\beta !}%
\right) ^{2}
\end{equation*}

\begin{lemma}
\label{lel2}Let $f=0,w=\sum_{\alpha }w_{\alpha }\mathfrak{N}_{\alpha }\in 
\mathcal{D}^{\prime }$.

(i) The solution to (\ref{12}) is given by%
\begin{equation}
u\left( t\right) =w\diamond \exp ^{\diamond }\left\{ \mathfrak{N}\left( \chi
_{\lbrack 0,t]}G\right) \right\} =\sum_{\alpha }\sum_{\beta \leq \alpha
}w_{\alpha -\beta }\frac{H(r)^{\beta }}{\beta !}\mathfrak{N}_{\alpha },0\leq
r\leq T.  \label{f5}
\end{equation}

(ii) If $w\in \mathcal{S}^{\prime }\left( \mathbf{R}\right) $, then $u\in C%
\mathcal{S}^{\prime }\left( [0,T];\mathbf{R}\right) ;$

(iii) 
\begin{equation*}
\sup_{t}\mathbf{E}\left[ u(t)^{2}\right] \leq \left\vert \left\vert
w\right\vert \right\vert ^{2},
\end{equation*}%
i.e. it is in $L^{2}\left( \Omega ,\mathbf{P}\right) $ if $\left\vert
\left\vert w\right\vert \right\vert <\infty $ (see Example \ref{exa1}\
below).

(iv) If $w=w_{0}$ is a constant, then $u(t)$ is adapted and%
\begin{equation*}
\sup_{t}\mathbf{E}\left[ u(t)^{2}\right] =w_{0}^{2}\exp \left\{ \int
\left\vert G\right\vert ^{2}d\mu \right\} .
\end{equation*}
\end{lemma}

\begin{proof}
(i) Let $M_{t}=\exp ^{\diamond }\left\{ \mathfrak{N}\left( \chi _{\lbrack
0,t]}G\right) \right\} $. By Corollary \ref{lr1}, 
\begin{equation*}
v(r)=w\diamond M_{r}=\sum_{\alpha }\sum_{\beta \leq \alpha }w_{\alpha -\beta
}\frac{H(r)^{\beta }}{\beta !}\mathfrak{N}_{\alpha },0\leq r\leq T.
\end{equation*}%
We will show that $v$ solves (\ref{12}). Indeed, 
\begin{eqnarray*}
&&\delta \left( \chi _{\lbrack 0,t]}vG\right) \\
&=&\sum_{|\alpha |\geq 1}\sum_{\left\vert p\right\vert
=1}\int_{0}^{t}\int_{V}\sum_{\beta \leq \alpha -p}w_{\alpha -p-\beta }\frac{%
H(r)^{\beta }}{\beta !}G\left( r,\upsilon \right) E_{p}(r,\upsilon )d\mu 
\mathfrak{N}_{\alpha } \\
&=&\sum_{|\alpha |\geq 1}\sum_{\left\vert p\right\vert
=1}\int_{U}\sum_{\beta \leq \alpha -p}w_{\alpha -p-\beta }\frac{\int \left(
\chi _{\lbrack 0,r]}G\right) ^{\otimes \left\vert \beta \right\vert
}E_{\beta }d\mu _{|\beta |}}{\left\vert \beta \right\vert !\beta !}\chi
_{\lbrack 0,t]}(r)G\left( r,\upsilon \right) E_{p}(r,\upsilon )d\mu 
\mathfrak{N}_{\alpha } \\
&=&\sum_{|\alpha |\geq 1}\sum_{\left\vert p\right\vert
=1}\int_{U^{\left\vert \gamma \right\vert }}\sum_{\gamma \leq \alpha
,\left\vert \gamma \right\vert \geq 1}w_{\alpha -\gamma }\left( \chi
_{\lbrack 0,r]}G\right) ^{\otimes \left\vert \gamma \right\vert -1}\chi
_{\lbrack 0,t]}(r)G\left( r,\upsilon \right) \frac{E_{\gamma }}{(\left\vert
\gamma \right\vert -1)!\gamma !}d\mu _{\left\vert \gamma \right\vert }%
\mathfrak{N}_{\alpha } \\
&=&\sum_{|\alpha |\geq 1}\sum_{\gamma \leq \alpha ,\left\vert \gamma
\right\vert \geq 1}\int w_{\alpha -\gamma }\left( \chi _{\lbrack
0,t]}G\right) ^{\otimes |\gamma |}\frac{E_{\gamma }}{|\gamma |!\gamma !}d\mu
_{\left\vert \gamma \right\vert }\mathfrak{N}_{\alpha }=v\left( t\right) -w,
\end{eqnarray*}%
and (\ref{12}) holds.

(ii) follows immediately by Lemma \ref{fl1} and Proposition \ref{prop6}. The
part (iii) is a direct consequence of (\ref{f5}). Finally, (iv) follows from
(\ref{f5}), Proposition \ref{prop6} and Corollary \ref{lr1}.
\end{proof}

\begin{example}
\label{exa1}Let $F\in L^{2}\left( U,d\mu \right) $. Taking $w=\exp
^{\diamond }\left\{ \mathfrak{N}\left( F\right) \right\} $ in (\ref{f5}), we
see that the solution to (\ref{12})%
\begin{eqnarray*}
u(t) &=&\exp ^{\diamond }\left\{ \mathfrak{N}\left( F\right) \right\}
\diamond \exp ^{\diamond }\left\{ \mathfrak{N}\left( \chi _{\lbrack
0,t]}G\right) \right\} \\
&=&\exp ^{\diamond }\left\{ \mathfrak{N}\left( F\right) +\mathfrak{N}\left(
\chi _{\lbrack 0,t]}G\right) \right\}
\end{eqnarray*}%
is clearly non-adapted in general but $\sup_{t}\mathbf{E}\left[ u(t)^{2}%
\right] <\infty $ (Proposition \ref{prop6}).
\end{example}

Let for $s\leq t,$ 
\begin{equation*}
H_{k}(s,t)=\int \chi _{\lbrack s,t]}Gm_{k}d\mu ,H(s,t)^{\alpha
}=\prod_{k}H_{k}(s,t)^{\alpha _{k}},\alpha \in J.
\end{equation*}%
For $f=\sum_{\alpha }f_{\alpha }\left( t,\upsilon \right) \mathfrak{N}%
_{\alpha }\in L_{2}\left( \mathcal{D}^{\prime }\left( U\right) ,d\mu \right) 
$, let $\left\vert \left\vert f\right\vert \right\vert _{0,T}^{2}=\mathbf{E}%
\int_{U}\left\vert f\right\vert ^{2}d\mu $ and 
\begin{eqnarray*}
\left\vert \left\vert f\right\vert \right\vert _{T}^{2}
&=&\sup_{t}\sum_{\alpha }\alpha !\left( \sum_{\left\vert p\right\vert
=1}\int_{0}^{t}\int_{U}\sum_{\beta +p\leq \alpha ,\left\vert \beta
\right\vert \leq n}f_{\alpha -p-\beta }(s,\upsilon )E_{p}(s,\upsilon )\frac{%
H(s,t)^{\beta }}{\beta !}d\mu \right) ^{2} \\
&&+\sum_{\alpha }\int_{U}\left\vert f_{\alpha }\right\vert ^{2}d\mu \alpha !.
\end{eqnarray*}

\begin{proposition}
\label{prop8}Let $~f=\sum_{\alpha }f_{\alpha }\left( t,\upsilon \right) 
\mathfrak{N}_{\alpha }\in L_{2}\left( \mathcal{D}^{\prime }\left( U\right)
,d\mu \right) ,w=\sum_{\alpha }w_{\alpha }\mathfrak{N}_{\alpha }\in \mathcal{%
D}^{\prime }.$

(i) The unique solution to (\ref{12}) in $C\mathcal{D}^{\prime }\left( [0,T];%
\mathbf{R}\right) $ can be given as 
\begin{eqnarray}
&&u\left( t\right)  \label{for0} \\
&=&w\diamond \exp ^{\diamond }\left\{ \mathfrak{N}\left( \chi _{\lbrack
0,t]}G\right) \right\} +\int_{0}^{t}\int \exp ^{\diamond }\left\{ \mathfrak{N%
}\left( \chi _{\lbrack s,t]}G\right) \right\} \diamond f\left( s\,,\upsilon
\right) \diamond \mathfrak{N(}ds,d\upsilon )  \notag \\
&=&\sum_{\left\vert \alpha \right\vert \geq 1}\sum_{\left\vert p\right\vert
=1}\int_{0}^{t}\int_{U}\sum_{\beta +p\leq \alpha ,\left\vert \beta
\right\vert \leq n}f_{\alpha -p-\beta }(s,\upsilon )E_{p}(s,\upsilon )\frac{%
H(s,t)^{\beta }}{\beta !}d\mu \mathfrak{N}_{\alpha }  \notag \\
&&+\sum_{\alpha }\sum_{\beta \leq \alpha }w_{\alpha -\beta }\frac{%
H(r)^{\beta }}{\beta !}\mathfrak{N}_{\alpha }.  \notag
\end{eqnarray}

(ii) The solution is the limit of Picards iterations $u^{n}\left( t\right) $%
: $u^{0}\left( t\right) =w+\int_{0}^{t}\int f(s,\upsilon )\diamond \mathfrak{%
N}\left( ds,d\upsilon \right) ,$%
\begin{equation}
u^{n+1}\left( t\right) =w+\int_{0}^{t}\int [u^{n}(s)G\left( s,\upsilon
\right) +f(s,\upsilon )]\diamond \mathfrak{N}\left( ds,d\upsilon \right)
,0\leq t\leq T.  \label{for00}
\end{equation}%
In fact,%
\begin{equation}
u^{n}(t)=w\diamond \sum_{k=0}^{n}\frac{\mathfrak{N}\left( \chi _{\lbrack
0,t]}G\right) ^{\diamond k}}{k!}+\int_{0}^{t}\int \sum_{k=0}^{n}\frac{%
\mathfrak{N}\left( \chi _{\lbrack s,t]}G\right) ^{\diamond k}}{k!}\diamond
f(s,\upsilon )\diamond \mathfrak{N}\left( ds,d\upsilon \right) .
\label{for1}
\end{equation}%
If $f\in \mathcal{S}^{\prime }\left( \mathbf{H,R}\right) $ and $w\in 
\mathcal{S}^{\prime }\left( \mathbf{R}\right) $, then $u^{n},u\in C\mathcal{S%
}^{\prime }\left( [0,T];\mathbf{R}\right) ;$

(iii) 
\begin{equation*}
\sup_{t}\mathbf{E}\left[ u\left( t\right) ^{2}\right] \leq 2\left(
\left\vert \left\vert w\right\vert \right\vert ^{2}+\left\vert \left\vert
f\right\vert \right\vert _{T}^{2}\right) .
\end{equation*}

(iv) If $w=w_{0}$ is deterministic and $f$ is adapted with $\left\vert
\left\vert f\right\vert \right\vert _{0,T}^{2}<\infty $, then $u\left(
t\right) $ is adapted and square integrable:%
\begin{equation*}
\sup_{t}\mathbf{E}\left[ u\left( t\right) ^{2}\right] \leq C\left( w_{0}^{2}+%
\mathbf{E}\int_{U}\left\vert f\right\vert ^{2}d\mu \right) .
\end{equation*}
\end{proposition}

\begin{proof}
Because of Lemma \ref{lel2} we assume $w=0$.

(i) Let 
\begin{eqnarray*}
l\left( s,\upsilon \right) &=&f\left( s,\upsilon \right) \diamond \exp
^{\diamond }\left\{ \mathfrak{N}\left( \chi _{\lbrack s,r]}\right) G\right\}
\\
&=&\sum_{\alpha }\sum_{\beta \leq \alpha ,\left\vert \beta \right\vert \leq
n}f_{\alpha -\beta }(s,\upsilon )\frac{H(s,r)^{\beta }}{\beta !},0\leq s\leq
r\leq T \\
&=&\sum_{\alpha }\sum_{\beta \leq \alpha ,\left\vert \beta \right\vert \leq
n}f_{\alpha -\beta }(s,\upsilon )\frac{1}{\left\vert \beta \right\vert
!\beta !}\int \left( \chi _{\lbrack s,t]}G\right) ^{\otimes \left\vert \beta
\right\vert }E_{\beta }d\mu _{\left\vert \beta \right\vert }\mathfrak{N}%
_{\alpha }.
\end{eqnarray*}%
and set%
\begin{eqnarray*}
&&v(r) \\
&=&\int_{0}^{r}\int l(s,\upsilon )\diamond \mathfrak{N}\left( ds,d\upsilon
\right) \\
&=&\sum_{\left\vert \alpha \right\vert \geq 1}\sum_{\left\vert p\right\vert
=1,p\leq \alpha }\int \sum_{\beta +p\leq \alpha ,\left\vert \beta
\right\vert \leq n}f_{\alpha -(p+\beta )}(s,\upsilon )E_{p}(s,\upsilon )d\mu 
\frac{H(s,t)^{\beta }}{\beta !}\mathfrak{N}_{\alpha },0\leq r\leq T.
\end{eqnarray*}%
For $r\in \lbrack 0,T],\upsilon =\left( s_{1},\upsilon _{1},\dots
,s_{k},\upsilon _{k}\right) \in U^{k},k\geq 1$, define%
\begin{eqnarray*}
\Phi \left( r,k,G,f\right) &=&\Phi \left( r,k,G,f\right) \left(
s_{1},\upsilon _{1},\dots ,s_{k},\upsilon _{k}\right) \\
&=&\sum_{j=1}^{k}f\left( \hat{s},\upsilon _{j}\right) \prod_{i\neq
j,i=1}^{k}\chi _{\left[ \hat{s},r\right] }\left( s_{i}\right) G\left(
s_{i},\upsilon _{i}\right) ,
\end{eqnarray*}%
where $\hat{s}=\min \left\{ s_{i},1\leq i\leq k\right\} .$

By Corollary \ref{lr1},%
\begin{eqnarray*}
&&v(r) \\
&=&\sum_{\left\vert \alpha \right\vert \geq 1,p\leq \alpha ,\left\vert
p\right\vert =1}\int_{0}^{r}\int \sum_{\beta +p\leq \alpha }f_{\alpha
-(p+\beta )}(s,\upsilon )\frac{E_{p}\left( s,\upsilon \right) }{\left\vert
\beta \right\vert !\beta !}\int \left( \chi _{\lbrack s,r]}G\right)
^{\otimes \left\vert \beta \right\vert }E_{\beta }d\mu _{\left\vert \beta
\right\vert }d\mu \mathfrak{N}_{\alpha } \\
&=&\sum_{\alpha }\sum_{\beta ^{\prime }\leq \alpha ,1\leq \left\vert \beta
^{\prime }\right\vert }\int_{U^{\left\vert \beta \right\vert +1}}\left( \chi
_{\lbrack s,r]}G\right) ^{\otimes \left\vert \beta ^{\prime }\right\vert
-1}f_{\alpha -\beta ^{\prime }}(s,\upsilon )\frac{E_{\beta ^{\prime }}}{%
(\left\vert \beta ^{\prime }\right\vert -1)!\beta ^{\prime }!}d\mu
_{\left\vert \beta ^{\prime }\right\vert }\mathfrak{N}_{\alpha } \\
&=&\sum_{\alpha }\sum_{\beta ^{\prime }\leq \alpha ,1\leq \left\vert \beta
^{\prime }\right\vert }\int_{U^{\left\vert \beta ^{\prime }\right\vert
}}\Phi \left( r,\left\vert \beta ^{\prime }\right\vert ,G,f_{\alpha -\beta
^{\prime }}\right) \frac{E_{\beta ^{\prime }}}{\left\vert \beta ^{\prime
}\right\vert !\beta ^{\prime }!}d\mu _{\left\vert \beta ^{\prime
}\right\vert }\mathfrak{N}_{\alpha },0\leq r\leq T.
\end{eqnarray*}%
We will show that $v$ solves (\ref{12}). Indeed, 
\begin{eqnarray*}
&&\int_{0}^{t}v(r)G(r,\upsilon )\diamond \mathfrak{N}\left( dr,d\upsilon
\right) \\
&=&\sum_{|\alpha |\geq 2}\sum_{\left\vert p\right\vert =1}\sum_{\beta
^{\prime }+p\leq \alpha ,1\leq \left\vert \beta ^{\prime }\right\vert }\int
\chi _{\lbrack 0,t]}(r)G(r,\upsilon )E_{p}\left( r,\upsilon \right) \times \\
&&\times \int_{U^{\left\vert \beta ^{\prime }\right\vert }}\Phi \left(
r,\left\vert \beta ^{\prime }\right\vert ,G,f_{\alpha -(p+\beta ^{\prime
})}\right) \frac{E_{\beta ^{\prime }}}{\left\vert \beta ^{\prime
}\right\vert !\beta ^{\prime }!}d\mu _{\left\vert \beta ^{\prime
}\right\vert }d\mu \mathfrak{N}_{\alpha } \\
&=&\sum_{|\alpha |\geq 2}\sum_{\left\vert p\right\vert =1}\sum_{\gamma
=\beta ^{\prime }+p\leq \alpha ,2\leq \left\vert \gamma \right\vert (\leq
n+2)}\int_{U^{\left\vert \gamma \right\vert }}\chi _{\lbrack
0,t]}(r)G(r,\upsilon )\Phi \left( r,\left\vert \gamma \right\vert
-1,G,f_{\alpha -\gamma }\right) \times \\
&&\times \frac{E_{\gamma }}{(\left\vert \gamma \right\vert -1)!\gamma !}d\mu
_{\left\vert \beta ^{\prime }\right\vert }d\mu \mathfrak{N}_{\alpha } \\
&=&\sum_{|\alpha |\geq 2}\sum_{\gamma =\beta ^{\prime }+p\leq \alpha ,2\leq
\left\vert \gamma \right\vert (\leq n+2)}\int_{U^{\left\vert \gamma
\right\vert }}\Phi \left( t,\left\vert \gamma \right\vert ,G,f_{\alpha
-\gamma }\right) \frac{E_{\gamma }}{\left\vert \gamma \right\vert !\gamma !}%
d\mu _{\left\vert \beta ^{\prime }\right\vert }d\mu \mathfrak{N}_{\alpha }
\end{eqnarray*}

and we see that%
\begin{equation*}
\int_{0}^{t}v(r)G(r,\upsilon )\diamond \mathfrak{N}\left( dr,d\upsilon
\right) =v(t)-\int_{0}^{t}\int f\left( r\right) \diamond G\left( r,\upsilon
\right) \mathfrak{N}\left( dr,d\upsilon \right) .
\end{equation*}

(ii) Consider $u^{n}\left( t\right) $ defined by (\ref{for1}) with $w=0$.
Then 
\begin{equation}
u^{n}(t)=\sum_{\left\vert \alpha \right\vert \geq 1}\sum_{\left\vert
p\right\vert =1,p\leq \alpha }\int \sum_{\beta +p\leq \alpha ,\left\vert
\beta \right\vert \leq n}f_{\alpha -(p+\beta )}(s,\upsilon )E_{p}(s,\upsilon
)d\mu \frac{H(s,t)^{\beta }}{\beta !}\mathfrak{N}_{\alpha },0\leq r\leq T,
\label{for2}
\end{equation}%
and repeating the proof of part (i) we see that for $,0\leq t\leq T$%
\begin{equation*}
\int_{0}^{t}\int u^{n}\left( r\right) G(r,\upsilon )\diamond \mathfrak{N}%
\left( dr,d\upsilon \right) =u^{n+1}\left( t\right) -\int_{0}^{t}\int
f\left( s,\upsilon \right) \diamond \mathfrak{N}\left( ds,d\upsilon \right) .
\end{equation*}%
If $f\in \mathcal{S}^{\prime }\left( \mathbf{H,R}\right) $, then $u^{0}\in C%
\mathcal{S}^{\prime }\left( [0,T];\mathbf{R}\right) $. If $u^{n}\in C%
\mathcal{S}^{\prime }\left( [0,T];\mathbf{R}\right) $, then $u^{n}G\in 
\mathcal{S}^{\prime }\left( \mathbf{H,R}\right) .$ By Proposition \ref{prop1}%
, $u^{n+1}\in C\mathcal{S}^{\prime }\left( [0,T];\mathbf{R}\right) $ and the
statement follows by comparing (\ref{for2}) and (\ref{for0}).

The part (iii) is a direct consequence of (\ref{for0}).

(iv) Since $\mathbf{E}\int \left\vert f\right\vert ^{2}d\mu <\infty ,$ it
follows that $f\in \mathcal{S}^{\prime }\left( \mathbf{H},\mathbf{R}\right) $%
, and according to part (ii) and Proposition \ref{prop4}, all the iterations
are adapted. Therefore Ito isometry holds. Obviously,%
\begin{equation*}
\sup_{t}\mathbf{E}\left[ u^{0}\left( t\right) ^{2}\right] \leq \mathbf{E}%
\int_{U}\left\vert f\right\vert ^{2}d\mu <\infty .
\end{equation*}%
Assume $\sup_{t}$ $\mathbf{E}\left[ u^{n}\left( t\right) ^{2}\right] <\infty 
$. Using (\ref{for00}) and Ito isometry,%
\begin{equation*}
\mathbf{E}\left[ u^{n+1}\left( t\right) ^{2}\right] \leq C[\int_{0}^{t}\int 
\mathbf{E}\left( u^{n}\left( s\right) ^{2}\right) G(s,\upsilon )^{2}d\mu ds+%
\mathbf{E}\int_{0}^{t}\int \left\vert f\right\vert ^{2}dsd\pi ]
\end{equation*}%
and by Gronwall's lemma there is a constant $C$ independent of $n$ such that%
\begin{equation*}
\sup_{n,t}\mathbf{E}\left( u^{n}\left( t\right) ^{2}\right) \leq C\mathbf{E}%
\int_{0}^{T}\int \left\vert f\right\vert ^{2}dsd\pi .
\end{equation*}%
Similarly, using Gronwall's lemma, we show that%
\begin{equation*}
\sum_{n}\sup_{t}\mathbf{E}\left( [u^{n+1}\left( t\right)
-u^{n}(t)]^{2}\right) <\infty .
\end{equation*}%
The statement follows.
\end{proof}

\subsubsection{Linear parabolic SPDEs}

In this section we extend the results on the linear SDE to a simple
parabolic SPDE.

Again, let $U=\left[ 0,T\right] \times V,d\mu =dtd\pi $. We denote $\mathbf{R%
}_{T}^{d}=\mathbf{R}^{d}\times \lbrack 0,T]$ and suppose that the following
measurable functions are given 
\begin{equation*}
a:\mathbf{R}^{d}\rightarrow \mathbf{R}^{d^{2}}\;,\quad b:\mathbf{R}%
^{d}\rightarrow \mathbf{R}^{d}.~
\end{equation*}

The following is assumed.

\textbf{A1}. Functions $a,b$, are infinitely differentiable and bounded with
all derivatives, and the matrix $a=\left( a^{ij}(x)\right) $ is symmetric
and non-degenerate: for all $x$%
\begin{equation*}
a^{ij}(x)\xi _{i}\xi _{j}\geq \delta \left\vert \xi \right\vert ^{2},\xi \in 
\mathbf{R}^{d},
\end{equation*}%
for some $\delta >0.$

Let $H_{2}^{s}=H_{2}^{s}\left( \mathbf{R}^{d}\right) ,s=1,2,$ be the Sobolev
class of square-integrable functions $v$ on $\mathbf{R}^{d}$ having
generalized space derivatives up to the $s$- order with the finite norm 
\begin{equation*}
|v|_{s,2}=|v|_{2}+|D_{x}^{s}v|_{2}\;,
\end{equation*}%
where $|v|_{2}=(\int_{\mathbf{R}^{d}}|v|^{2}\,dx)^{1/2}$.

Let $G\in L^{2}\left( [0,T]\times V,d\mu \right) $ with $d\mu =dtd\pi $. Let 
$L^{2,1}=L^{2,1}(\mathbf{R}^{d}\times \lbrack 0,T]\times V,dxdtd\pi $ be the
space of all measurable functions $g$ on $\mathbf{R}^{d}\times \lbrack
0,T]\times V$ such that%
\begin{equation*}
\left\vert \left\vert g\right\vert \right\vert _{1,2}^{2}=\int_{0}^{T}\int_{%
\mathbf{R}^{d}}\int_{V}[\left\vert g\left( s,x,\upsilon \right) \right\vert
^{2}+\left\vert \nabla _{x}g\left( s,x,\upsilon \right) \right\vert
^{2}]dsdxd\pi <\infty .
\end{equation*}%
Let $w=\sum_{\alpha }w_{\alpha }\left( x\right) \mathfrak{N}_{\alpha }\in 
\mathcal{D}^{\prime }\left( H_{2}^{3}(\mathbf{R}^{d})\right) $ and $%
f=\sum_{\alpha }f_{\alpha }(x,s,\upsilon )\mathfrak{N}_{\alpha }\in \mathcal{%
D}^{\prime }\left( L^{2,1}\right) $. The main objective of this section is
to study the equation for $u\left( t\right) =u\left( t,x\right) ,$%
\begin{eqnarray}
\partial _{t}u(x,t) &=&{\mathcal{L}}u(x,t)  \label{pe1} \\
&&+\int_{U}(u(x,t)G\left( t,\upsilon \right) +f(x,t,\upsilon ))\diamond 
\mathfrak{\dot{N}}\left( t,\upsilon \right) \pi \left( d\upsilon \right) 
\notag \\
u(0,x) &=&\varphi (x)\;,  \notag
\end{eqnarray}%
where $\mathcal{L}u=a^{ij}(x)u_{x_{i}x_{j}}+b^{i}(x){u_{x_{i}}.}$
Equivalently, we understand (\ref{pe1}) as 
\begin{eqnarray}
u(t) &=&w+\int_{0}^{t}{\mathcal{L}}u(s)ds  \label{pe2} \\
&&+\int_{0}^{t}\int_{U}\left[ u(s)G\left( s,\upsilon \right) +f\left(
s,\upsilon \right) \right] \diamond \mathfrak{N}\left( ds,d\upsilon \right) ,
\notag
\end{eqnarray}%
$0\leq t\leq T.$

{We will seek a solution to (\ref{pe2}) in the form }%
\begin{equation}
u(t)\mathbf{=}\sum_{\alpha }u_{\alpha }(t)\mathfrak{N}_{\alpha }\in C%
\mathcal{D}^{\prime }([0,T];H_{2}^{2}).  \label{pf1}
\end{equation}

We start our analysis of equation (\ref{pe2}) by introducing the definition
of a solution in the "weak sense".

\begin{definition}
{\label{defs}We say that a generalized }$\mathcal{D}${-process }$%
u(t)=\sum_{\alpha }u_{\alpha }(t)\xi _{\alpha }\in C\mathcal{D}^{\prime
}([0,T],H_{2}^{2})${\ is }$\mathcal{D}${-}$H_{2}^{2}${\ solution of equation
(\ref{pe2}) in }$[0,T]${, if the equality (\ref{pe2})} {holds in }$\mathcal{D%
}(L^{2}(\mathbf{R}^{d}))${\ for every }$0\leq t\leq T${. }
\end{definition}

\begin{lemma}
\label{rem12}{Assume \textbf{A1} holds, }$w(x)\mathbf{=}\sum_{\alpha
}w_{\alpha }(x)\mathfrak{N}_{\alpha }\in \mathcal{D}^{\prime }(H_{2}^{3})$,

$g=\sum_{\alpha }g_{\alpha }(x,s,\upsilon )\mathfrak{N}_{\alpha }\in 
\mathcal{D}^{\prime }\left( L^{2,1}\right) ${. }Then there is a unique
solution to (\ref{pe2}) in $C\mathcal{D}^{\prime }([0,T],H_{2}^{2})$ (Recall 
$C\mathcal{D}^{\prime }([0,T],H_{2}^{2})$ is the class of all generalized
processes $u=\sum_{\alpha }u_{\alpha }(t)\mathfrak{N}_{\alpha }$ on $\left[
0,T\right] $ such that $u_{\alpha }$ is $H_{2}^{2}$-valued continuous on $%
\left[ 0,T\right] $ $\forall \alpha \in J.$). The solution $u$ given by (\ref%
{pf1}) has the following coefficients: $u_{0}\left( t\right) =w_{0},$ 
\begin{equation}
\left\{ 
\begin{array}{l}
\partial _{t}u_{\alpha }(t)={\mathcal{L}}u_{\alpha
}+\sum_{k}\int_{V}m_{k}(u_{\alpha (k)}G+f_{\alpha \left( k\right) })d\pi \\ 
u_{\alpha }(0)=w_{\alpha }.%
\end{array}%
\right.  \label{pfo1}
\end{equation}
\end{lemma}

\begin{proof}
We seek the solution $u$ to (\ref{pe2}) in the form of (\ref{pf1}) with
continuous coefficients $u_{\alpha }$. Plugging the series (\ref{pf1}) into (%
\ref{pe2}) we immediately get the system (\ref{pfo1}). Indeed, by
definition, for $t\in \lbrack 0,T],$%
\begin{eqnarray*}
&&\sum_{\alpha }u_{\alpha }(x,t)\mathfrak{N}_{\alpha }=\sum_{\alpha
}w_{\alpha }\left( x\right) \mathfrak{N}_{\alpha } \\
&&+\sum_{\alpha }\int_{0}^{t}{\mathcal{L}}u_{\alpha }(x,s)ds\mathfrak{N}%
_{\alpha }+\sum_{\alpha }\sum_{k}\int_{0}^{t}\int_{V}m_{k}[u_{\alpha
(k)}G+g_{\alpha \left( k\right) }]d\pi ds\mathfrak{N}_{\alpha }.
\end{eqnarray*}%
Since the system is triangular, starting with $u_{0}\left( t\right) =w_{0}$
we find unique continuous $u_{\alpha }\left( t\right) $ for $\left\vert
\alpha \right\vert \geq 1$ (see \cite{la})$.$
\end{proof}

Denote by $T_{t}f$ the solution of the problem 
\begin{equation*}
\left\{ 
\begin{array}{l}
\partial _{t}u={\mathcal{L}}u,\quad 0\leq t\leq T, \\ 
u(0,x)=h(x),x\in \mathbf{R}^{d}.%
\end{array}%
\right.
\end{equation*}

\begin{remark}
\medskip \noindent \label{re6}\textit{If \textbf{A1 }}holds, then it is well
known that%
\begin{equation}
|T_{t}h|_{L^{2}(\mathbf{R}^{d})}^{2}\leq e^{Ct}|h|_{L^{2}(\mathbf{R}%
^{d})}^{2},  \label{pf:10}
\end{equation}%
(see \cite{la}).
\end{remark}

Note that for each $\alpha $, the solution $u_{\alpha }$ of (\ref{pfo1})
satisfies for $t\in \lbrack 0,T],$%
\begin{eqnarray}
u_{\alpha }(t) &=&T_{t}w_{\alpha
}+\sum_{k}\int_{0}^{t}\int_{V}[m_{k}(s,\upsilon )(T_{t-s}u_{\alpha
(k)}(s)G(s,\upsilon )  \label{pf11} \\
&&+T_{t-s}g_{\alpha \left( k\right) }(s,\upsilon ))]dsd\pi .  \notag
\end{eqnarray}

\begin{lemma}
\label{lel3}{Assume A1 holds, }%
\begin{equation*}
w(x)\mathbf{=}\sum_{\alpha }w_{\alpha }(x)\mathfrak{N}_{\alpha }\in \mathcal{%
D}^{\prime }(H_{2}^{3}),g=\sum_{\alpha }g_{\alpha }(x,s,\upsilon )\mathfrak{N%
}_{\alpha }\in \mathcal{D}^{\prime }\left( L^{2,1}\right) {.}
\end{equation*}%
{\ }Then $u$ is the unique solution to (\ref{pe2}) in $C\mathcal{D}^{\prime
}([0,T],H_{2}^{2})$ iff it is the unique solution to%
\begin{eqnarray}
u(t) &=&\int_{0}^{t}\int_{U}\left[ T_{t-s}u(s)G\left( s,\upsilon \right)
+T_{t-s}g\left( s,\upsilon \right) \right] \diamond \mathfrak{N}\left(
ds,d\upsilon \right)  \label{pf3} \\
&&+T_{t}w,  \notag
\end{eqnarray}%
$0\leq t\leq T.$
\end{lemma}

\begin{proof}
Since (\ref{pf11}) holds, the statement is an immediate consequence of Lemma %
\ref{rem12}.
\end{proof}

The following statement holds.

\begin{proposition}
\medskip \noindent \label{propf1}Let \textbf{A1 }hold and $w=\sum_{\alpha
}w_{\alpha }\left( x\right) \mathfrak{N}_{\alpha }\in \mathcal{D}^{\prime
}\left( H_{2}^{2}(\mathbf{R}^{d})\right) $ and $g=\sum_{\alpha }f_{\alpha
}(x,s,\upsilon )\mathfrak{N}_{\alpha }\in \mathcal{D}^{\prime }\left(
L^{2,1}\right) .$

(i) The unique solution to (\ref{pe2}) is given by 
\begin{eqnarray*}
u\left( t\right) &=&T_{t}w(x)\diamond \exp ^{\diamond }\left\{ \mathfrak{N}%
\left( \chi _{\lbrack 0,t]}G\right) \right\} \\
&&+\int_{0}^{t}\int \exp ^{\diamond }\left\{ \mathfrak{N}\left( \chi
_{\lbrack s,t]}G\right) \right\} \diamond T_{t-s}g\left( s,x,\upsilon
\right) \diamond \mathfrak{N(}ds,d\upsilon )
\end{eqnarray*}%
In the form of the series,%
\begin{eqnarray}
u(t) &=&\sum_{\left\vert \alpha \right\vert \geq 1}\sum_{\left\vert
p\right\vert =1}\int_{0}^{t}\int_{U}\sum_{\beta +p\leq \alpha
}T_{t-s}f_{\alpha -p-\beta }(s,\upsilon )E_{p}(s,\upsilon )\frac{%
H(s,t)^{\beta }}{\beta !}d\mu \mathfrak{N}_{\alpha }  \notag \\
&&+\sum_{\alpha }\sum_{\beta \leq \alpha }T_{t}w_{\alpha -\beta }\frac{%
H(t)^{\beta }}{\beta !}\mathfrak{N}_{\alpha }.  \notag
\end{eqnarray}

(ii) The solution is the limit of Picards iterations $u^{n}\left( t\right) $%
: $u^{0}\left( t\right) =T_{t}w+\int_{0}^{t}\int T_{t-s}f(s,\upsilon
)\diamond \mathfrak{N}\left( ds,d\upsilon \right) ,$%
\begin{equation*}
u^{n+1}\left( t\right) =T_{t}w+\int_{0}^{t}\int [T_{t-s}u^{n}(s)G\left(
s,\upsilon \right) +T_{t-s}f(s,\upsilon )]\diamond \mathfrak{N}\left(
ds,d\upsilon \right) ,
\end{equation*}%
$0\leq t\leq T.$ In fact, for $0\leq t\leq T,$%
\begin{eqnarray*}
u^{n}(t) &=&T_{t}w\diamond \sum_{k=0}^{n}\frac{\mathfrak{N}\left( \chi
_{\lbrack 0,t]}G\right) ^{\diamond k}}{k!} \\
&&+\int_{0}^{t}\int \sum_{k=0}^{n}\frac{\mathfrak{N}\left( \chi _{\lbrack
s,t]}G\right) ^{\diamond k}}{k!}\diamond T_{t-s}f(s,\upsilon )\diamond 
\mathfrak{N}\left( ds,d\upsilon \right) .
\end{eqnarray*}%
If $f\in \mathcal{S}^{\prime }\left( \mathbf{H,R}\right) $ and $w\in 
\mathcal{S}^{\prime }\left( \mathbf{R}\right) $, then $u^{n},u\in C\mathcal{S%
}^{\prime }\left( [0,T];\mathbf{R}\right) ;$

(iii) If $w$ is deterministic and $g$ is adapted, then the solution $u$ is $%
L^{2}\left( \mathbf{R}^{d}\right) $-valued and 
\begin{equation*}
\sup_{t}\mathbf{E}\left[ |u(t)|_{L^{2}\left( \mathbf{R}^{d}\right) }^{2}%
\right] \leq C\mathbf{E[}\left\vert w\right\vert _{L^{2}\left( \mathbf{R}%
^{d}\right) }+\int_{0}^{T}\int_{\mathbf{R}^{d}}\int_{U}\left\vert g\left(
x,s,\upsilon \right) \right\vert ^{2}dsdxd\pi ].
\end{equation*}
\end{proposition}

\begin{proof}
We repeat the main arguments of Proposition \ref{prop8} (as in the case of
linear SDE). The changes in the proof (i) are obvious. The proof of
(ii)-(iii) is identical to the proof of (ii), (iv) in Proposition \ref{prop8}
with the use of (\ref{pf:10}) for the estimate of the iterations $%
L^{2}\left( \Omega ,\mathbf{P}\right) $-norm.
\end{proof}

\subsection{\label{STPDE}Stationary SPDEs}

Let us consider a stationary (time independent) equation 
\begin{equation}
{\mathbf{A}}u+\mathbb{\delta }_{\mathfrak{\dot{N}}}({\mathbf{M}}u)=g
\label{eq:ellCopy2}
\end{equation}%
where, as previously, $\mathfrak{N}$-noise is a formal series $\mathfrak{%
\dot{N}}=\sum_{k}m_{k}\xi _{k}$ and $\left\{ m_{k}\right\} $ is a CONS in a
Hilbert space $H$ and $\xi _{k}$ are independent random variables with zero
mean and variance 1.

We will consider equation (\ref{eq:ellCopy2}) in a triple of Hilbert spaces $%
\left( V,H,V^{\prime }\right) ,$

\begin{itemize}
\item A. $V\subset H\subset V^{\prime }$ and the imbeddings $V\subset H$ and 
$H\subset V^{\prime }$ are dense and continuous;

\item B. The space $V^{\prime }$ is dual to $V$ relative to the inner
product in $H;$

\item C. There exists a constant $C>0$ such that $\left\vert
(u,v)_{H}\right\vert \leq C\left\Vert u\right\Vert _{V}\left\Vert
v\right\Vert _{V^{\prime }}$ for all $u$ and $v.$
\end{itemize}

A triple of Hilbert spaces is often call \textit{normal} if the assumptions
A,B,C hold. A typical example of a normal triple is the Sobolev spaces 
\begin{equation*}
\left( H^{l+\gamma }\left( \mathbb{R}^{d}\right) ,H^{l}\left( \mathbb{R}%
^{d}\right) ,H^{l-\gamma }\left( \mathbb{R}^{d}\right) \right) \text{ for }%
\gamma >0.
\end{equation*}

Everywhere in this section it is assumed that ${\mathbf{A}}:V\rightarrow
V^{\prime }$ and ${\mathbf{M}}:V\rightarrow V^{\prime }\otimes l_{2}$ are
bounded linear operators. \ 

As we already know, equation (\ref{eq:ellCopy2}) can be rewritten in the
form 
\begin{equation}
{\mathbf{A}}u+\sum_{n\geq 1}\mathbf{M}_{n}u\diamond \xi _{n}=f,
\label{eq:ell-b}
\end{equation}%
where $u=\sum_{\alpha }u_{\alpha }\mathfrak{N}_{\alpha }.$ Since 
\begin{equation}
\mathbf{M}_{n}u=\sum_{\alpha \mathbf{\in }{\mathcal{J}}}\mathbf{M}%
_{n}u_{\alpha }\mathfrak{N}_{\alpha },  \label{eq:operM}
\end{equation}%
we get 
\begin{eqnarray*}
&&\sum_{n\geq 1}\mathbf{M}_{n}u\diamond \xi _{n} \\
&=&\sum_{\alpha \mathbf{\in }{\mathcal{J}}}\sum_{n\geq 1}\mathbf{M}%
_{n}u_{\alpha }\xi _{\alpha }\diamond \xi _{n}+\sum_{\alpha \mathbf{\in }{%
\mathcal{J}}}\sum_{n\geq 1}\mathbf{M}_{n}u_{\alpha }\mathfrak{N}_{\alpha
+\varepsilon _{n}} \\
&=&\sum_{n\geq 1}\sum_{\beta \mathbf{\in }{\mathcal{J}}:\left\vert \beta
\right\vert \geq 1}\mathbf{M}_{n}u_{\beta -\varepsilon _{n}}\mathfrak{N}%
_{\beta }.
\end{eqnarray*}%
Therefore, for $\gamma \in J$ such that $\left\vert \gamma \right\vert >0$,
we have 
\begin{equation*}
\left( \sum_{n\geq 1}\mathbf{M}_{n}u\diamond \xi _{n}\right) _{\gamma
}=\sum_{n\geq 1}\mathbf{M}_{n}u_{\gamma -\varepsilon _{n}}
\end{equation*}

It is readily checked that the set $(u_{\alpha },$ $\alpha \in J)$ solves
the following system of deterministic equations related to (\ref{eq:ell-b})
is given by 
\begin{equation}
\begin{array}{c}
{\mathbf{A}}u_{\alpha }=Ef\text{ }\ \text{if }\left\vert \alpha \right\vert
=0 \\ 
\\ 
{\mathbf{A}}u_{\alpha }+\sum_{n\geq 1}\mathbf{M}_{n}u_{\alpha -\varepsilon
_{n}}=f_{\gamma }\text{ }\ \text{if }\left\vert \alpha \right\vert >0%
\end{array}
\label{prop1st}
\end{equation}

Note that the propagator (\ref{prop1st}) is lower triangular. Therefore, if $%
A$ has an appropriate inverse $A^{-1},$then the propagator can be solved
sequentially. Then, a solution to equation (\ref{eq:ell-b}) could be defined
by the following formula

\begin{equation}
u=\sum_{\alpha \in J}u_{\alpha }\mathfrak{N}_{\alpha }  \label{propSND}
\end{equation}%
where the sequence $\left\{ u_{\alpha },\alpha \in J\right\} .$

\bigskip Of course, an appropriate question to ask is: does equation (\ref%
{eq:ell-b}) has finite variance? The answer to this question is negative.
The following simple example clarifies this issue.

\begin{example}
\label{ExerFibon}. Consider the following simple version of equation%
\begin{equation}
u=1+u\diamond \xi .  \label{ex0}
\end{equation}%
\ Obviously, in this setting, $J=\left( 0,1,2,...\right) $ and consists of
one-dimensional indices$\ \alpha =0,1,2,...$ and $\mathbf{E}\left[ \mathfrak{%
N}_{i}^{2}\right] =i!$,

It is easy to see that $\left\{ u_{n}={\mathbb{E}}\left( u\mathfrak{N}%
_{n}\right) ,n\geq 0\;\right\} $ solves the following system of equations: 
\begin{equation*}
u_{0}=1,\;u_{n}=I_{n=0}+\sqrt{n}u_{n-1},\;n\geq 1
\end{equation*}%
Obviously, $u_{n}=\sqrt{n!}$and $v=1+\sqrt{n!}\mathfrak{N}_{n},$ Therefore, 
\begin{equation*}
{\mathbb{E}}u^{2}=\sum_{n\geq 1}u_{n}^{2}=\infty .
\end{equation*}
\end{example}

One could also define a solution to equation (\ref{eq:ell-b}) as a
generalized $\mathcal{D}${-random variable with values in }$V,${\ such that }%
(\ref{eq:ell-b}) holds in $\mathcal{D}(V^{\prime }).$

\subsubsection{\protect\bigskip Weighted Norms}

Another popular definition of solutions based on rescaling/weighting of the
coefficients $u_{\alpha }$ was discussed thoroughly in the literature on
polynomial chaos expansion for Gaussian and Levy processes (see, for
example, \cite{Ok2}, \cite{LR}, \cite{MR1}, \cite{NR}). This technique is
also suitable for the current setting and we will describe it briefly.

\bigskip Given a separable Hilbert space $X$ and sequence of positive
numbers $\mathit{R}=\left\{ r_{\alpha },\text{ }\alpha \in J\right\} ,$ we
define the space $\mathit{R}L_{2}\left( X\right) $ as the collection of
formal series $f=\sum_{\alpha }f_{\alpha }\mathfrak{N}_{\alpha },$ $%
f_{\alpha }\in X$ such that 
\begin{equation}
\left\Vert f\right\Vert _{\mathit{R}L_{2}\left( X\right) }^{2}=\sum_{\alpha
}\left\Vert f_{\alpha }\right\Vert _{X}^{2}r_{\alpha }^{2}<\infty
\label{normR}
\end{equation}%
If (\ref{normR}) holds, then $\sum_{\alpha }r_{\alpha }f_{\alpha }\xi
_{\alpha }\in L_{2}\left( X\right) $.

Similarly, the space $\mathit{R}^{-1}L_{2}\left( X\right) $ corresponds to
the sequence $\mathit{R}^{-1}=\left\{ r_{\alpha }^{-1},\text{ }\alpha \in
J\right\} .~$

Important and popular examples of the space $\mathit{R}L_{2}\left( X\right) $
correspond to the following weights:

(a) $r_{\alpha }^{2}=\Pi _{k=1}^{\infty }q_{k}^{\alpha _{k}},$ where $%
\left\{ q_{k},k\geq 1\right\} $ is a non-increasing sequence of positive
numbers;

(b) Kondratiev's spaces $\left( \mathit{S}\right) _{\rho ,\alpha }:$

\begin{equation*}
r_{\alpha }^{2}=\left( \alpha !\right) ^{\rho }\left( 2\mathbb{N}\right) ^{%
\mathit{l}\alpha }\text{ }\rho \leq 0,\text{ }\mathit{l}\leq 0.
\end{equation*}

\bigskip In particular, in the setting of Example \ref{ExerFibon}, ${\mathbb{%
E}}u^{2}=\left\Vert u\right\Vert _{\left( \mathit{S}\right)
_{0,0}}^{2}=\infty $, but $E\left\Vert u\right\Vert _{\left( \mathit{S}%
\right) _{\rho ,0}}^{2}<\infty $ for sufficiently small $\rho .$

\subsubsection{\label{WikNon}Wick-Nonlinear SPDEs}

Let us consider equation 
\begin{equation}
{\mathbf{A}}u-u^{\diamond 3}+\sum_{n\geq 1}\mathbf{M}_{n}u\diamond \xi _{n}=f%
\text{, }  \label{ell-3-b}
\end{equation}

where $u^{\diamond 3}=u\diamond u\diamond u.$ As in the previous section, we
will look for a chaos solution of the form%
\begin{equation*}
u=\sum_{\kappa \in J}u_{\kappa }\mathfrak{N}_{\kappa }.
\end{equation*}%
Obviously, 
\begin{equation*}
u^{\diamond 3}=\sum_{\kappa ,\beta ,\gamma \in J}u_{\kappa }u_{\beta
}u_{\gamma }\mathfrak{N}_{\kappa +\beta +\gamma }.
\end{equation*}%
Therefore, 
\begin{equation*}
\left( u^{\diamond 3}\right) _{\alpha }=\sum_{\kappa ,\beta ,\gamma :\kappa
+\beta +\gamma =\alpha }u_{\kappa }u_{\beta }u_{\gamma }
\end{equation*}%
and the propagator of equation (\ref{ell-3-b}) is given by 
\begin{equation}
{\mathbf{A}}u_{\alpha }-\sum_{\kappa ,\beta ,\gamma :\kappa +\beta +\gamma
=\alpha }u_{\kappa }u_{\beta }u_{\gamma }+\sum_{n\geq 1}\mathbf{M}%
_{n}u_{\alpha -\varepsilon _{n}}=f_{\alpha }  \label{eq:W2}
\end{equation}%
for all $\alpha \in J.$

\bigskip Similarly to (\ref{prop1st}),\ system (\ref{eq:W2}) is also lower
triangular and could be solved sequentially, assuming that operator $A$ has
an appropriate inverse.

\bigskip It is readily checked that if the Wick cubic $u^{\diamond 3}$ is
replaced by any Wick type polynomial then the related propagator system
remains to be lower triangular.

\begin{acknowledgement}
We express our gratitude to D. Nualart for a useful discussion.
\end{acknowledgement}

\section{Appendix}

Let $\left( \xi _{k}\right) $ be a sequence of r.v.. We assume that the
following assumption holds.

\textbf{G. }For each vector r.v. $\left( \xi _{i_{1}},\ldots ,\xi
_{i_{n}}\right) $ the moment generating function 
\begin{equation*}
M_{i_{1}\ldots i_{n}}(t)=M_{i_{1}\ldots i_{n}}(t_{1},\ldots ,t_{n})=\mathbf{E%
}\exp \left\{ t_{1}\xi _{i_{1}}+\ldots t_{n}\xi _{i_{n}}\right\}
\end{equation*}%
exists for all $t=\left( t_{1},\ldots ,t_{n}\right) $ in some neighborhood
of $0\in \mathbf{R}^{n}$.

Denote $\mathcal{J}$ the set of all multiinidices $\alpha =\left( \alpha
_{1},\alpha _{2},\ldots \right) $ such that $\left\vert \alpha \right\vert
<\infty $ and $\alpha _{k}\in \left\{ 0,1,2,\ldots \right\} .$ Let $\mathcal{%
G}=\sigma \left( \xi _{k},k\geq 1\right) $.

\begin{lemma}
\label{ale1}Let \textbf{G }holds. Then every $f\in L^{2}\left( \mathcal{G},%
\mathbf{P}\right) $ can be approximated in $L^{2}\left( \mathcal{G},\mathbf{P%
}\right) $ by a sequence of polynomials in $\xi ^{\alpha }=\prod_{k}\xi
_{k}^{\alpha _{k}},\alpha \in \mathcal{J}$.
\end{lemma}

\begin{proof}
The assumption \textbf{G} implies that all the moments of $\xi _{k},k\geq 1,$
exist. Assume $f\in L^{2}\left( \mathcal{G},\mathbf{P}\right) $ and 
\begin{equation}
\mathbf{E}\left[ f\xi ^{\alpha }\right] =0\text{ }\forall \alpha \in 
\mathcal{J}\text{.}  \label{2a}
\end{equation}%
It is enough to show that $f=0$ a.s. in such a case.

Let $\xi ^{n}=\left( \xi _{1},\ldots ,\xi _{n}\right) ,\theta =\left( \theta
_{1},\ldots ,\theta _{n}\right) \in \mathbf{R}^{n}$. Then%
\begin{equation*}
h\left( r,\theta \right) =\mathbf{E[}\exp \left\{ r\theta \cdot \xi
^{n}\right\} f]
\end{equation*}%
exists for $r\in \left( -\varepsilon ,\varepsilon \right) $ for some $%
\varepsilon >0$. Since $h\left( r\right) $ is a bilateral Laplace transform,
by Theorem 5a, p. 57 in \cite{w}, it must exist and be analytic for all
complex values of $r$ in the strip $-\varepsilon <\func{Re}r<\varepsilon .$
In addition, because of (\ref{2a}), 
\begin{equation*}
h\left( \iota u,\theta \right) =\mathbf{E[}\exp \left\{ \iota u\theta \cdot
\xi ^{n}\right\} f]=0
\end{equation*}%
for all $u\in \mathbf{R}$ (here $\iota ^{2}=-1).$ In particular, 
\begin{equation*}
\tilde{h}\left( \theta \right) =\mathbf{E[}\exp \left\{ \iota \theta \cdot
\xi ^{n}\right\} f]=0
\end{equation*}%
for all $\theta \in \mathbf{R}^{n}$. Therefore, 
\begin{equation}
\mathbf{E[}g(\xi ^{n})f]=0  \label{3}
\end{equation}%
for any continuous periodic function on $\mathbf{R}^{n}$. Then approximating
with long period functions we see that (\ref{3}) holds for a continuous
function on $\mathbf{R}^{n}$ with compact support, and then for any
continuous bounded function $g$ as well. Since $n$ is arbitrary, it follows
that $f=0$ a.s.
\end{proof}

\begin{lemma}
\label{fl1}For $v=\sum_{\left\vert \alpha \right\vert =n}v_{\alpha
}E_{\alpha }\in \mathbf{H}^{\hat{\otimes}n}(Y)$ and $u=\sum_{\left\vert
\alpha \right\vert =n}u_{\alpha }E_{\alpha }\in \mathbf{H}^{\hat{\otimes}%
n}(Y),$ we have%
\begin{equation*}
I_{n}\left( v\right) \diamond I_{m}\left( u\right) =I_{n+m}\left( \widetilde{%
v\otimes _{Y}u}\right) ,
\end{equation*}%
where $v\otimes _{Y}u=\left( v\left( x\right) ,u\left( y\right) \right)
_{Y},x\in U^{n},y\in U^{m}.$
\end{lemma}

\begin{proof}
Indeed, $I_{n}\left( v\right) =n!\sum_{\left\vert \alpha \right\vert =n}v_{a}%
\mathfrak{N}_{\alpha },~I_{m}\left( u\right) =m!\sum_{\left\vert \alpha
\right\vert =m}u_{\alpha }\mathfrak{N}_{\alpha }$ and%
\begin{eqnarray*}
I_{n}\left( v\right) \diamond I_{m}\left( u\right) &=&n!m!\sum_{\alpha
}\sum_{\beta \leq \alpha }\left( v_{\beta },u_{\alpha -\beta }\right) 
\mathfrak{N}_{\alpha } \\
&=&n!m!\sum_{|\alpha |=n+m}\sum_{\beta \leq \alpha }\int \left( v\otimes
_{Y}u\right) \frac{E_{\beta }}{\beta !n!}\frac{E_{\alpha -\beta }}{\left(
\alpha -\beta \right) !m!}\left( v_{\beta },u_{\alpha -\beta }\right) 
\mathfrak{N}_{\alpha } \\
&=&\sum_{\alpha }\sum_{\beta \leq \alpha }\int \left( v\otimes _{Y}u\right) 
\frac{E_{\beta }}{\beta !}\frac{E_{\alpha -\beta }}{\left( \alpha -\beta
\right) !}\mathfrak{N}_{\alpha }=\sum_{\alpha }\int \left( v\otimes
_{Y}u\right) \frac{E_{\alpha }}{\alpha !}\mathfrak{N}_{\alpha } \\
&=&(n+m)!\sum_{\left\vert \alpha \right\vert =n+m}l_{\alpha }\mathfrak{N}%
_{\alpha }=I_{n+m}\left( \widetilde{v\otimes _{Y}u}\right) ,
\end{eqnarray*}%
because $\widetilde{v\otimes _{Y}u}=\sum_{\left\vert \alpha \right\vert
=n+m}l_{\alpha }E_{\alpha }$ with 
\begin{equation*}
l_{\alpha }=\frac{1}{\alpha !(n+m)!}\int (v\otimes _{Y}u)E_{\alpha }d\mu
_{n+m}.
\end{equation*}
\end{proof}

\end{document}